\documentclass[a4paper, reqno, 14pt]{amsart}

\usepackage[utf8]{inputenc}
\usepackage{amsthm,amsfonts,amssymb,amsmath}
\usepackage{longtable}
\usepackage[T1]{fontenc}
\usepackage{enumitem}
\usepackage{verbatim}
\usepackage{pdfsync}
\usepackage{mathtools}
\usepackage{mathrsfs}
\usepackage{stmaryrd}
\usepackage{tikz-cd}

\usetikzlibrary{babel}
\usepackage[colorlinks=
						citecolor=Black,
						linkcolor=Red,
						urlcolor=Blue]{hyperref}

\newtheorem{theorem}{Theorem}[section]
\newtheorem{proposition}[theorem]{Proposition}
\newtheorem{lemma}[theorem]{Lemma}

\newtheorem{corollary}[theorem]{Corollary}

\theoremstyle{definition}
\newtheorem{definition}[theorem]{Definition}

\newtheorem{remark}[theorem]{Remark}
\newtheorem{remarks}[theorem]{Remarks}

\numberwithin{equation}{section}

\newenvironment{altenumerate}
   {\begin{list}
      {\textup{(\theenumi)} }
      {\usecounter{enumi}
       \setlength{\labelwidth}{0pt}
       \setlength{\labelsep}{0pt}
       \setlength{\leftmargin}{0pt}
       \setlength{\itemsep}{\the\smallskipamount}
       \renewcommand{\theenumi}{\roman{enumi}}
      }}
   {\end{list}}
\newenvironment{altitemize}
   {\begin{list}
      {$\bullet$}
      {\setlength{\labelwidth}{0pt}
	   \setlength{\itemindent}{5pt}
       \setlength{\labelsep}{5pt}
       \setlength{\leftmargin}{0pt}
       \setlength{\itemsep}{\the\smallskipamount}
      }}
   {\end{list}}

\numberwithin{equation}{section}

\newcommand{\BA}{{\mathbb {A}}}

\newcommand{\BF}{{\mathbb {F}}}

\newcommand{\BN}{{\mathbb {N}}}

\newcommand{\BP}{{\mathbb {P}}}

\newcommand{\CA}{{\mathcal {A}}}
\newcommand{\CB}{{\mathcal {B}}}

\newcommand{\CF}{{\mathcal {F}}}

\newcommand{\CI}{{\mathcal {I}}}
\newcommand{\CJ}{{\mathcal {J}}}

\newcommand{\CL}{{\mathcal {L}}}

\newcommand{\CO}{{\mathcal {O}}}
\newcommand{\CP}{{\mathcal {P}}}
\newcommand{\CQ}{{\mathcal {Q}}}

\newcommand{\CU}{{\mathcal {U}}}

\newcommand{\RX}{{\mathrm {X}}}

\newcommand{\Ru}{{\mathrm {u}}}

\newcommand{\bB}{{\mathrm{\textbf{B}}}}

\newcommand{\bP}{{\mathrm{\textbf{P}}}}
\newcommand{\bQ}{{\mathrm{\textbf{Q}}}}

\newcommand{\Fm}{{\mathfrak {m}}}

\newcommand{\Fp}{{\mathfrak {p}}}
\newcommand{\Fq}{{\mathfrak {q}}}

\newcommand{\FP}{{\mathfrak {P}}}

\newcommand{\Bl}{{\mathrm{Bl}}}
\newcommand{\bigslant}[2]{{\left.{#1}\middle/{#2}\right.}}

\renewcommand{\dim}{{\mathrm{dim}}}
\newcommand{\defeq}{\vcentcolon=}

\newcommand{\eqdef}{=\vcentcolon}

\newcommand{\Gal}{{\mathrm{Gal}}}

\newcommand{\GL}{{\mathrm{GL}}}

\newcommand{\Ker}{{\mathrm{Ker}}}

\newcommand{\lra}{\longrightarrow}

\newcommand{\ov}[1]{\overline{#1}} 

\newcommand{\pr}{{\mathrm{pr}}}
\newcommand{\Proj}{{\mathrm{Proj \;}}}
\newcommand{\PGL}{ {\mathrm{PGL} } }

\newcommand{\res}[1]{{\!\,\mid_{#1}}}

\newcommand{\Spec}{{\mathrm{Spec \;}}}

\newcommand{\Span}{{\mathrm{span}}}
\newcommand{\nsubset}{{\not\subset}}

\newcommand{\unts}{^{\times}}
\newcommand{\ul}[1]{{\underline{#1}}}

\newcommand{\wt}[1]{{\widetilde{#1}}}

\newcommand{\alggen}[3]{{ \Delta^{#1}_{{#2} \setminus {#3}} }}
\newcommand{\idealgen}[3]{{ \Delta^{#1}_{{#2} \setminus {#3}} }}

\let\originalmiddle\middle
\renewcommand{\middle}[1]{\,\originalmiddle#1\,}

\newlength{\arrow}
\newlength{\mysize}
\newcommand*{\myrightarrow}[1]{\xrightarrow{\mathmakebox[\mysize]{#1}}}

\newcommand{\myhookrightarrow}[1]{ \xhookrightarrow{\mathmakebox[\mysize]{#1}}}

\newlength{\mysizee}
\newcommand{\myrightarrowdbl}[1]{\xrightarrow{ \mathmakebox[\mysize]{#1} } \mathrel{\mkern-20.5mu}\rightarrow\mathrel{\mkern+2.5mu}}

\title{Compactifications of the Drinfeld half space over a Finite Field}

\author{Georg Linden}

\address{Fachgruppe Mathematik und Informatik, Bergische Universität Wuppertal, Gaußstrasse 20, 42119 Wuppertal, Germany}
\email{linden@math.uni-wuppertal.de}

\date{April 18, 2018}

\begin{document}

\begin{abstract}
	When considered as a Deligne--Lusztig variety, the Drinfeld half space $\Omega_V$ over a finite field $k$ has a compactification whose boundary divisor is normal crossing and which can be obtained by successively blowing-up projective space along linear subspaces.
	Pink and Schieder \cite{PS14} have introduced a new compactification of $\Omega_V$ whose strata of the boundary are glued together in a way dual to the way they are for the tautological compactification by projective space.
	We show that by applying an analogous sequence of blow-ups to this new compactification we arrive at the compactification by Deligne and Lusztig as well.
	Moreover, we compute for each of these three compactifications the stabilizers of $\bar{k}$-valued points under the canonical $\PGL(V)$-action.
	We find that in each case the stratification can be recovered from the unipotent radicals of these stabilizers.
\end{abstract}

\maketitle

\tableofcontents

\section{Introduction}

Let $V$ be a finite-dimensional vector space over a finite field $k$.
We use the convention that $\bP_V = \Proj\!( \mathrm{Sym} \, V ) $ denotes the projective space of hyperplanes in $V$.
The Drinfeld half space $\Omega_V$ is defined as the open subvariety of $\bP_V$ consisting of the hyperplanes that do not contain any $k$-rational line of $V$.
In other words, $\Omega_V$ is the complement of all $k$-rational hyperplanes
\[ \Omega_V = \bP_V \setminus \; \bigcup_{\mathclap{\substack{W \subset V, \\ \dim_k \, W =1 }}} \; \bP_{V/W} , \]
and hence in particular affine.

When considered as a subvariety of a complete flag variety, the Drinfeld half space is a basic example of a Deligne--Lusztig variety \cite[2.2]{DL76}.
As such, it has a compactification as defined in \cite[9.10]{DL76} which we denote by $\bar{\Omega}_V$.
This compactification is not isomorphic to the tautological compactification $\Omega_V \subset \bP_V$; nevertheless, it is closely related to $\bP_V$.
Let $\wt{\bP}_V$ denote the variety obtained by successively blowing-up along the strict transforms of $k$-linear subspaces of $\bP_V$ with increasing dimension.
Then $\bar{\Omega}_V$ is isomorphic to $\wt{\bP}_V$ \cite[Prop.\ 9.1]{L17}, \cite[4.1.2]{W14}.
The variety $\wt{\bP}_V$ is also used in \cite{RTW13} to show that every $k$-automorphism of $\Omega_V$ extends to a $k$-automorphism of $\bP_V$:
First the authors prove that every $k$-automorphism of $\Omega_V$ extends to a $k$-automorphism of $\wt{\bP}_V$ and in a second step that such an automorphism descends to a $k$-automorphism of $\bP_V$.
Moreover, in \cite{L17} the author uses a different approach in the first step to show that every separable and dominant $k$-endomorphism of $\Omega_V$ extends to a $k$-automorphism of $\wt{\bP}_V$.

Another compactification of $\Omega_V$ was introduced and investigated by Schieder \cite{S09}, and Pink and Schieder \cite{PS14}:
Let $R_V$ denote the sub-$k$-algebra of the field of fractions $\mathrm{Frac} \, (\mathrm{Sym}\,V)$ generated by the elements $\frac{1}{v}$, for $v\in V \setminus \{0\}$, i.e.\
\[ R_V = k \! \left[ \frac{1}{v} \middle| v \in V\setminus \{0\} \right] .\]
Setting $\deg (\frac{1}{v}) =1 $ gives $R_V$ the structure of a graded $k$-algebra, and one defines
\[ \bQ_V = \Proj R_V .\]
Then the Drinfeld half space $\Omega_V$ embeds into $\bQ_V$ as the open affine subvariety obtained by inverting all elements $\frac{1}{v} \in R_V$, for $v\in V\setminus \{0\}$.

We consider these three compactifications in more detail in section \ref{Sect - Comparison of Compactifications}.
We recall the definition of a Deligne--Lusztig variety, describe $\Omega_V$ as such, and mention some properties of the variety $\bQ_V$ that are shown in \cite{PS14}.
Furthermore, we state the natural stratifications of $\bP_V$ and $\bQ_V$ which are indexed by quotients respectively subspaces of $V$.
We proceed by constructing a variety $\wt{\bQ}_V$ by an analogous process of successive blow-ups of $\bQ_V$.
Our main result of this section is that there exists an isomorphism $\wt{\bP}_V \cong \wt{\bQ}_V$ which is the identity when restricted to $\Omega_V$ (Theorem \ref{Thm - P Q Isomorphism}).
This extends a result by Schieder \cite[Thm.\ 5.2]{S09} who proved this for $\dim_k \, V \leq 3$.
Moreover, we describe the natural stratification of $\bar{\Omega}_V$ which is indexed by flags of $V$.

In Section \ref{Sect - Modular interpretation}, we begin by recalling the modular interpretation of $\bQ_V$ that Pink and Schieder gave \cite[Ch.\ 7]{PS14}:
Similarly to the way that $\bP_V$ represents the functor of invertible sheaves with generating global sections, the $k$-scheme $\bQ_V$ represents a functor of invertible sheaves and so called reciprocal maps.
Pink and Schieder also construct a compactification of $\Omega_V$ which dominates both $\bP_V$ and $\bQ_V$ and is a desingularization of the latter one by defining a functor $B_V$ related to those of $\bP_V$ and $\bQ_V$, and proving that it is representable.
We complement this by showing that this functor $B_V$ is in fact the functor of points of $\wt{\bP}_V$.
Hence $\bB_V \defeq \wt{\bP}_V \cong \wt{\bQ}_V$ is isomorphic to the desingularization by Pink and Schieder.
Let $\bar{k}$ be an algebraic closure of $k$.
For $\bar{k}$-valued points of $\bP_V$, $\bQ_V$ and $\bB_V$, we also describe in which stratum they lie, in terms of their modular interpretation.

The set of $\bar{k}$-valued points of each of $\bP_V$, $\bQ_V$ and $\bB_V$ comes with a canonical $\PGL(V)$-action (where $\PGL(V) = \mathbf{PGL}(V)(k)$).
In Theorem \ref{Thm - Stabilizers}, we determine the stabilizers under this $\PGL(V)$-action for each $\bar{k}$-valued point in all three cases.
We conclude in Corollary \ref{Cor - Stabilizers and Stratification} that the stratifications of $\bP_V$, $\bQ_V$ and $\bB_V$ can be recovered from the unipotent radicals of the stabilizers of its $\bar{k}$-valued points.

I would like to thank M.\ ~Rapoport for suggesting this very interesting topic as well as for his guidance and advice.
I also would like to thank S.\ ~Orlik for helpful remarks.

\section{Comparison of Compactifications}\label{Sect - Comparison of Compactifications}

\subsection{$\Omega_V$ as a Deligne--Lusztig variety}

Let us begin with a brief recollection of the definition of a Deligne--Lusztig variety and its compactification as in \cite{DL76}.
Let $k= \BF_q$ be a finite field, and fix an algebraic closure $\bar{k}$ of $k$.
Let $G$ be a connected reductive algebraic group over $\bar{k}$, obtained by extensions of scalars from $G_0$ over $k$, and denote by $F \colon G \lra G$ the corresponding Frobenius endomorphism.
Fix a maximal torus $T$ of $G$ and a Borel subgroup $B$ containing $T$ and denote the associated Weyl group by $W = N(T)/T$.
Let $X$ be the set of all Borel subgroups of $G$.
Then $G$ acts on $X$ by conjugation, and this gives a natural isomorphism
\begin{align*}
	G/B	\myrightarrow{\sim}	X	, \;	gB	\longmapsto		g B g^{-1} 
\end{align*}
which gives $X$ the structure of a projective algebraic variety.
We can identify the Weyl group $W$ with the set of orbits of $G$ in $X\times X$.
Namely for $w\in W$, let $O(w) \defeq G \ldotp ( B, \dot{w}B \dot{w}^{-1} ) $ where $\dot{w} \in N(T)$ is a representative of $w$.
Two Borel subgroups $B', B''$ of $G$ are said to be \textit{in relative position $w$}, for $w\in W$, if $\left(B',B'' \right) \in O(w)$.

Then the \textit{Deligne--Lusztig variety} $X(w) \subset X$ is defined as the locally closed subvariety of $X$ consisting of all Borel subgroups $B$ of $G$ such that $B$ and $F(B)$ are in relative position $w$.
Note that the action of $G$ on $X$ induces an action of $G^{F}= G_0(k)$ on $X(w)$.

Deligne and Lusztig construct a compactification of $X(w)$ in the following way \cite[9.10]{DL76}:
Consider a minimal expression $w=s_1 \cdots s_m$ in the Weyl group.
Then
\begin{equation}\label{Eq - Compactification of DL-variety}
\begin{aligned}
	\bar{X}(w) \defeq& \left\{ \left( B_0,\ldots,B_r \right) \in X^{r+1} \middle|\begin{array}{l}
																					B_r = F(B_0), \forall 1 \leq i \leq r \colon  \\
																					\left( B_{i-1}, B_{i} \right) \in O(s_i) \text{ or } \left( B_{i-1}, B_{i} \right) \in O(e)
																				\end{array}    \right\}											\\
		\cup	\;\qquad	&\qquad \qquad \qquad \qquad   \hspace{19.2pt} \cup																\\
	X(w) \, \cong&  \left\{ \left( B_0,\ldots,B_r \right) \in X^{r+1} \middle|\begin{array}{l}
																		B_r = F(B_0), \forall 1 \leq i \leq r \colon  \\
																		\left( B_{i-1}, B_{i} \right) \in O(s_i)
																		\end{array}    \right\}	
\end{aligned}
\end{equation}
is a smooth projective compactification of $X(w)$ (depending on the minimal expression for $w$) and the boundary divisor $\bar{X}(w)\setminus X(w)$ has normal crossings \cite[Lemma 9.11]{DL76}.

Let $V$ be a $k$-vector space of dimension $n+1 \geq 2$.
To match the terminology of \cite{DL76}, we consider the Drinfeld half space in the equivalent dual situation, i.e.\ in the projective space of lines $\bP_{V^{\ast}}$ we consider the open subvariety $\Omega_{V^{\ast}}$ of lines that are not contained in any $k$-rational hyperplane.

Then the Drinfeld half space $\Omega_{V^{\ast}}$ occurs as the following example of a Deligne--Lusztig variety:
Let $G =\mathbf{GL}(V_{\bar{k}})= \mathbf{GL}(V)_{\bar{k}}$.
For a basis $e_1,\ldots,e_{n+1}$ of $V$, choose the group of diagonal matrices with respect to this basis as maximal torus $T\subset G$ and the group of upper triangular matrices as $B \subset G$.
As an element of the Weyl group $W \cong \mathfrak{S}_{n+1}$ of $G$ consider $w = (1,\ldots, n+1)$.
We can represent $W$ by the subgroup of permutation matrices in $N(T)$, and the fundamental reflections are the transpositions $(i, i+1)$, $i=1,\ldots,n$.
Furthermore, $X$ is the variety $\mathrm{Flag}_{V_{\bar{k}}}$ of complete flags $\left( \{0\}\subset V_1 \subset\ldots\subset V_{n} \subset V_{\bar{k}} \right)$ of $V_{\bar{k}}$.
For $w= (1,\ldots, n+1)$, two flags $\CF$ and $\CF'$ are in relative position $w$ if and only if 
\[ V_i^{\prime} + V_i = V_{i+1} \, \text{, for $i = 1,\ldots n$, and } \, V_{n}^{\prime} + V_{1} = V_{\bar{k}} .\]
For $\CF$ and $\CF' = F(\CF)$, this is equivalent to the condition that
\[ V_{i} = \bigoplus_{j=0}^{i-1} F^j (V_1) \, \text{, for $i=1,\ldots,n+1$, } \]
where we write $V_{n+1}=V_{\bar{k}}$.
Now, consider the projection
\begin{align}\label{Eq - complete flag projection proj space}
	\mathrm{Flag}_{V_{\bar{k}}} \lra \bP_{V^{\ast}}, \; \left( \{0\}\subset V_1 \subset\ldots\subset V_{n} \subset V_{\bar{k}} \right) \longmapsto V_1 .
\end{align}
Then this map induces an isomorphism between $\mathrm{Flag}_{V_{\bar{k}}}(w)$ and $\Omega_{V^{\ast}}$.

We will denote the compactification of $\mathrm{Flag}_{V_{\bar{k}}}(w)$ according to \eqref{Eq - Compactification of DL-variety} with respect to the minimal expression $w=(1,2)\cdots (n,n+1)$ by $\bar{\Omega}_{V^{\ast}}$.
We remark that the closure of $\Omega_{V^{\ast}}$ embedded into $\mathrm{Flag}_{V_{\bar{k}}}$ via $\Omega_{V^{\ast}} \cong \mathrm{Flag}_{V_{\bar{k}}}(w)$ is already isomorphic to $\bar{\Omega}_{V^{\ast}}$ because $w$ is a Coxeter element \cite[Lemma 1.9]{H99}.

\subsection{$\bP_V$ as a compactification of $\Omega_V$}

Of course, by the definition of the Drinfeld half space, the projective space $\bP_{V}$ also constitutes a compactification of $\Omega_{V}$; then $\Omega_V$ is the dense open subset
\[\Omega_{V} = \Spec k \! \left[ \frac{v}{w} \middle| v \in V, w \in V\setminus \{0\} \right] \subset \Proj k \! \left[ v \middle| v \in V \right] = \bP_{V}, \]
i.e.\ the open subset where the function $\prod_{v \in V\setminus \{0\}} v$ does not vanish.
However, the boundary-divisor
\[\bP_{V}\setminus \Omega_V = \; \bigcup_{\mathclap{\substack{W \subset V, \\ \dim_k \, W =1 }}} \; \bP_{V/W}\]
does not have normal crossings for $\dim_k \, V = n+1 > 2$: 
Let $\bP_{V/V'}$, with $\dim_k \, V' = n$, be a $k$-rational point of $\bP_V$.
Then the irreducible components $\bP_{V/W}$ of the boundary-divisor which intersect in $\bP_{V/V'}$ are those with $W \subset V'$ and there are exactly
\[ \binom{n}{n-1}_q =  q^{n-1} +\ldots+ q +1 > n \]
of them.

\subsection{$\bQ_V$ as a compactification of $\Omega_V$}

Another compactification of $\Omega_V$ has been investigated by Schieder in \cite{S09}, and Pink and Schieder in \cite{PS14}.
It is defined as
\[ \bQ_{V} = \Proj R_V , \]
where
\[R_V = k \! \left[ \frac{1}{v} \middle| v \in V\setminus \{0\} \right]\]
is viewed as a sub-$k$-algebra of the field of fractions $ \mathrm{Frac} \, k \! \left[ v \middle| v\in V \right] $, endowed with the grading $\deg \frac{1}{v} = 1$.
Then the Drinfeld half space embeds into $\bQ_{V}$ as the dense open subset where the function $\prod_{v \in V\setminus \{0\}} \frac{1}{v} $ does not vanish, i.e.\
\[ \Omega_{V} = \Spec k \!\left[ \frac{v}{w} \middle| v \in V, w \in V\setminus \{0\} \right] \subset \bQ_{V}. \]
We mention some properties of the variety $\bQ_V$:
The only relations between the generators of $R_V$ are
\begin{align*}
	\frac{1}{\lambda v} &= \lambda^{-1} \frac{1}{v}	\, ,					&\,&\text{ for all $v\in V\setminus \{0\}$, $\lambda \in k$, and}			\\
	\frac{1}{v} \frac{1}{v'} &= \frac{1}{v+v'}\left( \frac{1}{v} + \frac{1}{v'} \right) ,	&\,&\text{ for all $v, v' \in V\setminus \{0\}$, such that $v+v' \neq 0$}
\end{align*}
\cite[Thm.\ 1.6]{PS14}.
This gives an embedding of $\bQ_V$ as a closed subset of $ \BP_k^{q^{n+1}-2}$.
Furthermore, $\bQ_V$ is integral, Cohen-Macaulay and projectively normal \cite[Thm.\ 1.11]{PS14}.
Moreover, $\bQ_V$ is not regular for $\dim_k \, V > 2$; it is regular along all strata of codimension $\leq 1$, but singular along all strata of codimension $\geq 2$ \cite[Thm.\ 8.4]{PS14}.
The boundary-divisor of $\bQ_{V}$ does not have normal crossings for the same reason as for $\bP_V$.

\subsection{Stratification of $\bP_V$ and $\bQ_V$}

The varieties $\bP_{V}$ and $\bQ_{V}$ both come with a stratification by locally closed subsets whose combinatorics are dual to each other:
For a subspace $V' \subsetneq V$ there is a closed immersion $\bP_{V/V'}  \myhookrightarrow{} \bP_{V}$ which corresponds to the homogeneous ideal
\[\left( v \middle| v \in V' \right) \subset k \! \left[ v \middle| v\in V \right] . \]
We identify $\bP_{V/V'}$ with its image in $\bP_{V}$ and call it a \textit{linear subspace}.
We remark that $\dim \, \bP_{V/V'} = n - \dim_k \, V'$.
Then the stratum associated to $V'$ is defined as
\[ \CP_{V/V'} \defeq \bP_{V/V'} \setminus \; \bigcup_{\mathclap{V' \subsetneq W \subsetneq V}} \; \bP_{V/W} \subset \bP_{V}. \]
Observe that $\CP_{V/V'} = \Omega_{V/V'} \subset \bP_{V/V'}$, in particular $\CP_{V}= \Omega_{V}$.
Furthermore, the closure of $\CP_{V/V'}$ in $\bP_V$ is the disjoint union of all strata $\CP_{V/W}$, for $V' \subset W \subsetneq V$:
\[ \ov{\CP_{V/V'} } = \; \bigsqcup_{\mathclap{V' \subset W \subsetneq V}} \; \CP_{V/W} . \]

Similarly, for a subspace $\{0\} \neq V' \subset V$ one finds that
\[\left( \frac{1}{v} \middle| v \in V \setminus V' \right) \subset k \! \left[ \frac{1}{v} \middle| v \in V\setminus \{0\} \right]\]
is a homogeneous ideal \cite[Cor. 4.5]{S09} which induces a closed immersion $\bQ_{V'} \myhookrightarrow{} \bQ_V$.
Again, we identify $\bQ_{V'}$ with its image in $\bQ_V$ and call it a \textit{linear subspace}. Then $\dim \, \bQ_{V'} = \dim_k \, V' -1$.
We associate to $V'$ the stratum
\[ \CQ_{V'} \defeq \bQ_{V'} \setminus \; \bigcup_{\mathclap{ \{0\} \neq W \subsetneq V'}} \; \bQ_{W} \subset \bQ_V. \]
Then $\CQ_{V'} = \Omega_{V'} \subset \bQ_{V'}$, and the closure of $\CQ_{V'}$ in $\bQ_V$ is the disjoint union of all strata $\CQ_{W}$, for $\{0\}\neq W \subset V'$:
\[ \ov{\CQ_{V'}} = \; \bigsqcup_{\mathclap{\{0\}\neq W \subset V'}} \; \CQ_{W} . \]

\subsection{Blow-up sequences for $\bP_V$ and $\bQ_V$}

To relate $\bP_V$ and $\bQ_V$ to the compactification $\bar{\Omega}_V$ by Deligne and Lusztig, we consider the following construction, cf.\ \cite[Ch.\ ~8]{L17}, \cite[Def.\ ~1.2]{RTW13}:	
	
\begin{definition}\label{Def - P blow-up arrangement}	
	We define the \textit{blow-up of $\mathrm{\mathbf{P}}_V$ along the full arrangement of linear subspaces} to be
	\begin{align*}
		\pi \colon \wt{\bP}_V = \wt{\bP}_{n-1} \myrightarrow{\pi_{n-1}} \wt{\bP}_{n-2} \myrightarrow{\pi_{n-2}} \ldots \myrightarrow{\pi_2} \wt{\bP}_1 	\myrightarrow{\pi_1} \wt{\bP}_0 \myrightarrow{\pi_0} \wt{\bP}_{-1} = \bP_V ,
	\end{align*}
	where $\pi_i$ denotes the blow-up of $\wt{\bP}_{i-1}$ along the strict transforms of all linear subspaces $L \subset \bP_V$ with $\dim \, L = i$.
\end{definition}

Note that $\pi_{n-1}$ is an isomorphism. Furthermore, all $\pi_i$ induce an isomorphism on $\Omega_V$ and $\Omega_V \subset \wt{\bP}_V$ is a dense open subset.
Moreover, every linear subspace $\bP_{V/V'} \subset \bP_V$ gives rise to an exceptional divisor $D_{V'} \subset \wt{\bP}_V$ under the blow-up $\pi$.
Note that the divisors $D_{V'}$ and $D_{V''}$ have non-empty intersection if and only if $V' \subset V''$ or $V'' \subset V'$; otherwise the intersection $\bP_{V/V'} \cap \bP_{V/V''}$ is a linear subspace of strictly smaller dimension.
Then the strict transform of the intersection is blown-up before the ones of $\bP_{V/V'}$ or $\bP_{V/V''}$, separating those.

\begin{definition}
	Similarly, we define the \textit{blow-up of $\mathrm{\mathbf{Q}}_V$ along the full arrangement of linear subspaces} to be
	\begin{align*}
		\rho \colon \wt{\bQ}_V = \wt{\bQ}_{n-1} \myrightarrow{\rho_{n-1}} \wt{\bQ}_{n-2} \myrightarrow{\rho_{n-2}} \ldots \myrightarrow{\rho_2} \wt{\bQ}_1 \myrightarrow{\rho_1} \wt{\bQ}_0 \myrightarrow{\rho_0} \wt{\bQ}_{-1} = \bQ_V ,
	\end{align*}
	where $\rho_i$ denotes the blow-up of $\wt{\bQ}_{i-1}$ along the strict transforms of all linear subspaces $H \subset \bQ_V$ with $\dim \, H = i$.
\end{definition}

Again, $\rho_{n-1}$ is an isomorphism, all $\rho_i$ induce an isomorphism on $\Omega_V$, and $\Omega_V \subset \wt{\bQ}_V$ is a dense open subset.
Also, every linear subspace $\bQ_{V'} \subset \bQ_V$ gives rise to an exceptional divisor $E_{V'}\subset \wt{\bQ}_V$ under the blow-up $\rho$.
Like before, the intersection of the divisors $E_{V'}$ and $E_{V''}$ is non-empty if and only if $V' \subset V''$ or $V'' \subset V'$.

Then, in \cite[Prop.\ 9.1]{L17} and \cite[4.1.2]{W14} the following is shown:

\begin{theorem}
	The compactification $\bar{\Omega}_V$ is isomorphic to $\wt{\mathrm{\mathbf{P}}}_V$, extending the isomorphism $\bar{\Omega}_V \supset \mathrm{Flag}_V(w) \cong \Omega_V \subset \wt{\mathrm{\mathbf{P}}}_V$ induced by \eqref{Eq - complete flag projection proj space}.
\end{theorem}

Extending a result from Schieder \cite[Thm.\ 5.2]{S09} for $\dim_k \, V \leq 3$, we will add to the picture:
	
\begin{theorem}\label{Thm - P Q Isomorphism}
	There exists an isomorphism $\wt{\mathrm{\mathbf{P}}} \cong \wt{\mathrm{\mathbf{Q}}}$, necessarily unique, which extends the identity on $\wt{\mathrm{\mathbf{P}}} \supset \Omega = \Omega \subset \wt{\mathrm{\mathbf{Q}}}$.
	Furthermore, for each $\{0\}\neq V' \subsetneq V$, the exceptional divisors obtained from $\mathrm{\mathbf{P}}_{V/V'}$ under $\pi$ and from $\mathrm{\mathbf{Q}}_{V'}$ under $\rho$ agree.
\end{theorem}

The proof of this theorem will be done in section \ref{Subsect - Proof of Thm P Q Isomorphism}.
We conclude that $\wt{\bP}_V$, $\wt{\bQ}_V$, and $\bar{\Omega}_V$ are isomorphic compactifications of $\Omega_V$.
From now on, we denote this compactification which dominates both $\bP_V$ and $\bQ_V$ by $\bB_V$.

\subsection{Stratification of $\bB_V$}

For $\{0\}\neq V' \subsetneq V$, we write $D_{V'} \subset \bB_V$ for the exceptional divisor obtained from $\bP_{V/V'}$ respectively $\bQ_{V'}$.
This gives a stratification of $\bB_V$ as follows:
For an arbitrary flag $\CF $ of $V$ let
\[ \CB_{\CF} \defeq \bigg( \;\, \bigcap_{\mathclap{W \in \CF}} \; D_{W} \bigg) \mathbin{\bigg\backslash} \bigg( \;\, \bigcup_{\mathclap{W \notin \CF}} \; D_{W} \bigg) \]
where we use the convention that the trivial flag $(\{0\}\subset V)$ corresponds to 
\[ \CB_{(\{0\}\subset V)} \defeq \bB_{V} \mathbin{\big\backslash} \bigg( \;\;\;\;\, \bigcup_{\mathclap{ \{0\}\neq W \subsetneq V }} \; D_{W} \bigg) = \Omega_V .\]
Let $\CF \leq \CF'$ signify that the flag $\CF'$ is a refinement of $\CF$ (including $\CF' = \CF$).
As two divisors intersect if and only if there is an inclusion between their corresponding subspaces, the closure of the stratum $\CB_{\CF}$ is given by the disjoint union
\[ \ov{\CB_{\CF}} = \bigsqcup_{\CF \leq \CF'} \CB_{\CF'} . \]

We conclude this section by describing $\bB_V$ similarly to the way, in which the blow-up of the origin $P \in \BA^n_k$ is the subvariety of $\BA^n_k \times \BP^{n-1}_k$ defined by the vanishing of the polynomials
\[ x_i y_j - x_j y_i, \,\, i,j=1,\ldots,n\]
(where $\BA^n_k = \Spec k \! \left[ x_1,\ldots,x_n \right]$ and $\BP^{n-1}_k = \Proj k \! \left[y_1,\ldots,y_n \right]$).

We define
\[ \bP_{\ul{V},i} \defeq \;\; \prod\limits_{\mathclap{\substack{W \subset V, \\ \dim_k \, W \geq n-i}}} \;\;\; \bP_{W} \]
so that $\bP_{\ul{V},-1}= \bP_V$.
Then 
\[\bP_{\ul{V},i} = \mathrm{multiProj} \; \bigg( \;\;\bigotimes\limits_{\mathclap{\substack{W \subset V, \\ \dim_k \, W \geq n-i}}} \;\;\; k \! \left[ v\middle| v\in W \right] \bigg)\]
where 
\[\;\; \bigotimes\limits_{\mathclap{\substack{W \subset V, \\ \dim_k \, W \geq n-i}}} \;\;\; k \! \left[ v\middle| v\in W \right] \]
is a multi-graded $k$-algebra with a grading for each $W \subset V$ with $ \dim_k \, W \geq n-i$, and we\linebreak take the $\mathrm{multiProj}$-construction with respect to these gradings.
We have for the structure sheaf
\[ \CO_{\bP_{\ul{V},i}} = \;\; \bigotimes\limits_{\mathclap{\substack{W \subset V, \\ \dim_k \, W \geq n-i}}} \;\;\; \pr^{\ast}_{\bP_W} \CO_{\bP_{W}} .  \]
For $ V' \subset V$, with $\dim_k \, V' \geq n-i$, we consider $\pr^{\ast}_{\bP_W} \CO_{\bP_V'} \subset \CO_{\bP_{\ul{V},i}}$ as a subsheaf.
Any ideal sheaf $\CI \subset \CO_{\bP_{V'}}$ induces an extended ideal sheaf in $\CO_{\bP_{\ul{V},i}}$ which we will also denote by $\CI \subset \CO_{\bP_{\ul{V},i}}$ when this causes no confusion.

For $V'' \subset V' \subset V$, with $\dim_k \, V'' \geq n-i$, consider the following ideal sheaf:
\begin{align}\label{Eq - Generators of incidence ideal sheaf}
	\left( v \otimes v' - v' \otimes v \middle| v,v' \in V'' \right)^{\sim} \subset \pr_{V'}^{\ast} \CO_{\bP_{V'}} \otimes \pr_{V''}^{\ast}\CO_{\bP_{V''}} .
\end{align}
Define $\CJ_i \subset \CO_{\bP_{\ul{V},i}}$ to be the sum of the extensions of the ideal sheaves above, for all such $V'' \subset V' \subset V$.
Let $\RX_{i} \subset \bP_{\ul{V},i}$ be the subvariety defined by $\CJ_i$.
We call a subvariety obtained this way an \textit{incidence variety}, and an ideal generated by relations as in \eqref{Eq - Generators of incidence ideal sheaf} an \textit{incidence ideal}.
Furthermore, note that the projection $\bP_{\ul{V},i} \lra \bP_{\ul{V},i-1}$ induces a projection $\RX_{i} \lra \RX_{i-1}$.

\begin{proposition}\label{Prop - B Incidence variety X_i isomorphism}
	For all $i=0,\ldots,n-1$, let $\wt{\mathrm{\mathbf{P}}}_i$ denote the $i$-th step in the blow-up of $\mathrm{\mathbf{P}}_V$ along the full arrangement of linear subspaces as in Definition \ref{Def - P blow-up arrangement}.
	Then there exists an isomorphism 
	\[b_i \colon \wt{\mathrm{\mathbf{P}}}_i \myrightarrow{\sim} \RX_i  \]
	such that the following diagram commutes
	\[\begin{tikzcd}
		\wt{\mathrm{\mathbf{P}}}_i		\arrow[r, "b_i"] \arrow[d, "\pi_i"]	& \RX_i \arrow[d] \\
		\wt{\mathrm{\mathbf{P}}}_{i-1}	\arrow[r, "b_{i-1}"] 				& \RX_{i-1}  . 
	\end{tikzcd}\]
	In particular, $\mathrm{\mathbf{B}}_V \cong \RX_{n-1} \subset \mathrm{\mathbf{P}}_{\ul{V},n-1}$.
	Moreover, under $b_i$ the exceptional divisor $D_{V'}\subset \wt{\mathrm{\mathbf{P}}}_i$, for $V' \subsetneq V$ with $\dim_k \, V' \geq n-i$, is identified with
	\[\RX_i \cap  \Bigg(  \;\;\;\;  \prod\limits_{\mathclap{\substack{V' \subsetneq W \subset V}}} \; \mathrm{\mathbf{P}}_{W/V'} \times \mathrm{\mathbf{P}}_{V'} \times \;\; \prod\limits_{\mathclap{\substack{V' \nsubset W \subset V, \\ \dim_k \, W \geq n-i}}} \; \mathrm{\mathbf{P}}_{W} \Bigg) \subset \RX_i ,\]
	where for $V' \subsetneq W$ we consider $\mathrm{\mathbf{P}}_{W/V'}$ as a closed subset of $\mathrm{\mathbf{P}}_{W}$.
\end{proposition}

We will prove this proposition in section \ref{Subsect - Proof of Prop B Incidence variety X_i isomorphism}.

\subsection{Proof of Theorem \ref{Thm - P Q Isomorphism}}\label{Subsect - Proof of Thm P Q Isomorphism}

	Our aim is to construct a covering of $\wt{\bP}_V$ by open affine subsets $U_{\CF}$ and a covering of $\wt{\bQ}_V$ by open affine subsets $W_{\CF}$, both indexed by all complete flags $\CF$ of $V$, such that $\Omega_V \subset U_{\CF}$ and $\Omega_V \subset W_{\CF}$, for all flags $\CF$.
	We will show that for a fixed complete flag $\CF$ there is an isomorphism $U_{\CF} \cong W_{\CF}$ which extends the identity on $\Omega_V$.
	As $\wt{\bP}_V$ and $\wt{\bQ}_V$ are reduced and separated over $k$, these isomorphisms glue to a unique isomorphism $\wt{\bP}_V \cong \wt{\bQ}_V$ extending the identity on $\Omega_V$.

	Fix a complete flag $\CF = \left( \{0\}= V_n \subsetneq V_{n-1} \subsetneq \ldots \subsetneq V_0 \subsetneq V_{-1} = V \right)$ of $V$.
	We write $\CF_i$ for the flag $\CF_i = \left( \{0\} \subset V_{i+1} \subsetneq \ldots \subsetneq V \right)$ obtained from $\CF$ by truncating from below (so that $\CF = \CF_{n-1}$).
	Let $e_0,\ldots, e_n$ be a basis of $V$ such that $V_i = \Span \left( e_{i+1},\ldots, e_n \right) $ and let $L_i = \bP_{V/ V_i}$ be the linear subspace associated to $V_i$.
	This way, we obtain a chain $L_0 \subset \ldots \subset L_{n-1} \subset \bP_V$.
	By induction, we will construct open affine subsets $U_{\CF_i} \subset \wt{\bP}_i$, for $i= -1, \ldots, n-1$, such that $\pi_{i}(U_{\CF_i}) \subset U_{\CF_{i-1}}$ and $\Omega_V \subset U_{\CF_{i}}$.
	Afterwards, we will show that the subsets $U_{\CF}$ indeed cover $\wt{\bP}_V$.

	For the initial step, we consider the principal open subset $U_{\CF_{-1}} \defeq \Spec A_{\CF_{-1}} \subset \wt{\bP}_{-1}$ where we obtain $A_{\CF_{-1}}$ as the localization of $k \! \left[ v \middle| v\in V \right] $ at the element $\prod_{v\in V\setminus V_0} v$.
	
	To simplify our notation in the following, we write
	\[ \Delta_{S'}^{S} \defeq \left\{ \frac{v}{w} \middle| v \in S, w\in S' \right\} \subset
	\mathrm{Frac} \, k \! \left[ v \middle| v \in V \right]  ,\]
	for subsets $S\subset V$ and $S' \subset V\setminus \{0\}$.
	
	Then $A_{\CF_{-1}}$ is the $k$-algebra with generators $\alggen{V}{V}{V_0} $, i.e.\
	\[ A_{\CF_{-1}} = k \! \left[ \alggen{V}{V}{V_0} \right] , \]
	and $L_j \cap U_{-1}$, for $j=0,\ldots,n-1$, is given by the ideal
	\[ I_{\CF_{-1}}^{(j)} \defeq \left( \idealgen{V_j}{V}{V_0} \right) \subset A_{\CF_{-1}} . \]	

	For the induction step, suppose that we have already constructed the open subset 
	\[ U_{\CF_{i-1}} = \Spec A_{\CF_{i-1}} \subset \wt{\bP}_{i-1} \]
	where
	\[ A_{\CF_{i-1}} = k \! \left[ \alggen{V}{V}{V_0}, \alggen{V_0}{V_0}{V_1}, \ldots ,\alggen{V_{i-1}}{V_{i-1}}{V_i} \right] . \]
	Furthermore, assume that the strict transform of $L_j$, for $j\geq i $, in $U_{\CF_{i-1}}$ is given by the ideal
	\[ I_{\CF_{i-1}}^{(j)} = \left( \idealgen{V_j}{V_{i-1}}{V_i} \right) \subset A_{\CF_{i-1}} , \]
	and the exceptional divisor of $L_j$, for $j\leq i-1$, i.e.\ the preimage of $L_j$ in $U_{\CF_{i-1}}$ under the blow-up $\pi_{i-1} \circ \ldots \circ \pi_0$ (an empty assumption if $i=0$), is given by the ideal $\big( \frac{e_{j+1}}{e_j} \big) \subset A_{\CF_{i-1}}$.

\begin{lemma}\label{BlowupexplicitLemmaforP}
	All strict transforms of linear subspaces $L'= \mathrm{\mathbf{P}}_{V/V'} \subset \mathrm{\mathbf{P}}_V$ with $\dim \, L' \geq i$ and $L_i \nsubset L'$ have empty intersection with $U_{\CF_{i-1}}$.
\end{lemma}

	We postpone the proof of the this and all coming lemmata of the subsection, as well as the proofs of the numbered equations until the end of this subsection.
	With the above lemma, we see that the centre of $\pi_{i} \colon \pi_{i}^{-1} ( U_{\CF_{i-1}} ) \lra U_{\CF_{i-1}}$ consists only of the strict transform of $L_i$ in $U_{\CF_{i-1}}$.
	Hence we find that
	\[ \pi_{i}^{-1} ( U_{\CF_{i-1}} ) = \Proj \Bl_{I_{\CF_{i-1}}^{(i)}} A_{\CF_{i-1}} ,  \]
	where
	\[\Bl_{I_{\CF_{i-1}}^{(i)}} A_{\CF_{i-1}} = A_{\CF_{i-1}} \oplus I_{\CF_{i-1}}^{(i)} \oplus \left( I_{\CF_{i-1}}^{(i)} \right)^2 \oplus \ldots \,  . \]
	We consider the principal open subset $U_{\CF_i} \defeq \Spec A_{\CF_i} \subset \pi_{i}^{-1} ( U_{\CF_{i-1}} )$, where $A_{\CF_i}$ is the coordinate ring, consisting of the elements of degree $0$ of the localization of $\Bl_{I_{\CF_{i-1}}^{(i)}} A_{\CF_{i-1}}$ at the element
	\[ g_{V_{i-1}\setminus V_{i}}^{V_{i} \setminus V_{i+1}} \defeq \;\;\;\; \prod_{\mathclap{\substack{ v\in V_{i} \setminus V_{i+1} \\ w\in V_{i-1} \setminus V_{i} }} } \;\;\; \frac{v}{w} . \]
	We compute that
	\begin{align}\label{BlowupcoordinateringforPEquation}
		A_{\CF_i} = \Big( \Bl_{I_{\CF_{i-1}}^{(i)}} A_{\CF_{i-1}} \Big)_{\left[ g_{V_{i-1}\setminus V_{i}}^{V_{i} \setminus V_{i+1}} \right]_0 } 
			= A_{\CF_{i-1}} \! \left[ \alggen{V_i}{V_i}{V_{i+1}} \right] .
	\end{align}
	Furthermore for $j \geq i$, we find that the total transform of $L_j$ in $U_{\CF_i}$ is
	\begin{align}\label{BlowuptotaltransformforPEquation}
		I_{\CF_{i-1}}^{(j)}  A_{\CF_i} = \bigg( \frac{e_{i+1}}{e_i} \bigg) \bigg( \idealgen{V_j}{V_i}{V_{i+1}} \bigg) \subset A_{\CF_i} .
	\end{align}
	Hence, the exceptional divisor of $\pi_i$ in $U_{\CF_i}$, i.e.\ the case $j=i$, is given by $\big( \frac{e_{i+1}}{e_i} \big) \subset A_{\CF_i} $. Therefore the strict transform of $L_j$ in $U_{\CF_i}$, for $j \geq i+1$, is given by
	\[ I_{\CF_i}^{(j)} \defeq \left( \idealgen{V_j}{V_i}{V_{i+1}} \right) . \]
	Note that the ideal $\big( \frac{e_{i+1}}{e_i} \big) \subset A_{\CF_i} $ does not depend on the choice of the basis $e_0,\ldots,e_n$ of $V$ with $V_i = \Span \left( e_{i+1} ,\ldots, e_n \right)$.
	This finishes the construction of $U_{\CF_i} \subset \wt{\bP}_i$.
	
	It remains to show that the sets $U_{\CF} = U_{\CF_{n-1}}$ cover $\wt{\bP}_V$ if $\CF
	$ runs through all complete flags of $V$.
	It suffices to do this stepwise, via the following lemma:
	
\begin{lemma}\label{BlowupcoverLemmaforP}
	For fixed $i \in \{0, \ldots, n-1\}$, let $U_{\CF_{i-1}} = \Spec A_{\CF_{i-1}} \subset \wt{\mathrm{\mathbf{P}}}_{i-1}$ be the open subset corresponding to the flag $ \CF_{i-1} = \left( \{0\}\subset V_{i} \subsetneq \ldots \subsetneq V \right) $ with $\dim_k \, V_i = n-i$.
	Then, as $ \CF_{i}' = \left( \{0\} \subset V_{i+1}' \subsetneq V_{i} \subsetneq\ldots\subsetneq V \right) $ ranges over all such flags with $\dim_k \, V_{i+1}' = n-i-1$, the corresponding open subsets $U_{\CF_i'} = \Spec A_{\CF_i'}$ cover $\pi_{i}^{-1} \left( U_{\CF_{i-1}} \right) \subset \wt{\mathrm{\mathbf{P}}}_{i}$.
	Here, we write
	\begin{align*}
		A_{\CF_i'}  \defeq A_{\CF_{i-1}} \! \left[ \alggen{V_i}{V_i}{V_{i+1}'} \right] = k \! \left[ \alggen{V}{V}{V_0} , \ldots, \alggen{V_{i-1}}{V_{i-1}}{V_{i}} ,  \alggen{V_{i}}{V_{i}}{V_{i+1}'} \right] .
	\end{align*}
	Furthermore, the open subsets $U_{\CF_{-1}'} = \Spec A_{\CF_{-1}'}$, defined analogously, each contain $\Omega_V$, and a covering of $\mathrm{\mathbf{P}}_V$.
\end{lemma}

	To construct the covering of $\wt{\bQ}_V$, we proceed in an analogous way.
	Fix again a complete flag $\CF = \left( \{0\}= V_n \subsetneq V_{n-1} \subsetneq \ldots \subsetneq V_0 \subsetneq V_{-1} = V \right)$ of $V$ and let $\CF^i = \left( \{0\} \subsetneq\ldots\subsetneq V_{n-i-2} \subset V \right)$, now truncated from above.
	Like before, choose a subordinate basis $e_0,\ldots, e_n$ of $V$ such that $V_i = \Span \left( e_{i+1},\ldots, e_n \right) $.
	Now consider the linear subspaces $H_i = \bQ_{V_{n-i-1}}$.
	Note that $\dim \, H_i = i$ and that we have $H_0 \subset \ldots \subset H_{n-1} \subset \bQ_V$.
	Again, we will construct affine open subsets $\Omega_V \subset W_{\CF^i} \subset \wt{\bQ}_i$, for $i=-1,\ldots, n-1$, inductively and prove that the sets $W_\CF = W_{\CF^{n-1}}$ cover $\wt{\bQ}_V$.
	
	We begin by localizing $k \! \left[ \frac{1}{v} \middle| v\in V \setminus\{0\} \right]$ at the element $\prod_{v \in V_{n-1}\setminus \{0\}} \frac{1}{v}$ to obtain 
	\[B_{\CF^{-1}} \defeq k \! \left[ \alggen{V_{n-1}}{V}{\{0\}} \right] \]
	which defines the principal open subset $W_{\CF^{-1}} \defeq \Spec B_{\CF^{-1}} \subset \wt{\bQ}_{-1}$.
	Then $H_j \cap W_{\CF^{-1}}$ is given by the ideal
	\[ J^{(j)}_{\CF^{-1}} \defeq \left( \idealgen{V_{n-1}}{V}{V_{n-j-1}} \right) \subset B_{\CF^{-1}} , \]
	for $j=0, \ldots, n-1$.
	
	To proceed with the induction, suppose we have constructed $W_{\CF^{i-1}} = \Spec B_{\CF^{i-1}} \subset \wt{\bQ}_{i-1}$ with
	\[ B_{\CF^{i-1}} = k \! \left[ \alggen{V_{n-1}}{V_{n-1}}{\{0\}}, \ldots, \alggen{V_{n-i}}{V_{n-i}}{V_{n-i+1}}, \alggen{V_{n-i-1}}{V}{V_{n-i}} \right] , \]
	and that the strict transform of $H_j$ in $W_{\CF^{i-1}}$, for $j\geq i $, is given by the ideal 
	\[J_{\CF^{i-1}}^{(j)} = \left( \idealgen{V_{n-i-1}}{V}{V_{n-j-1}} \right) \subset B_{\CF^{i-1}}. \]
	Also assume that the exceptional divisor of $H_j$, for $j\leq i-1$, i.e.\ the preimage of $H_j$ in $W_{\CF^{i-1}}$ under the blow-up $\rho_{i-1} \circ\ldots\circ \rho_0$ (again, an empty assumption if $i=0$), is given by $ \big( \frac{e_{n-j}}{e_{n-j-1}} \big) \subset B_{\CF^{i-1}} $.

\begin{lemma}\label{BlowupexplicitLemmaforQ}
	All strict transforms of linear subspaces $H' = \mathrm{\mathbf{Q}}_{V'} \subset \mathrm{\mathbf{Q}}_V$ with $\dim \, H' \geq i$ and $H_i \nsubset H'$ have empty intersection with $W_{\CF^{i-1}}$.
\end{lemma}
	
	Hence the centre of $\rho_i \colon \rho_i^{-1} ( W_{\CF^{i-1}} ) \lra W_{\CF^{i-1}} $ consists only of the strict transform of $H_i$ in $W_{\CF^{i-1}}$ so that
	\[ \rho_i^{-1} ( W_{\CF^{i-1}} ) = \Proj \Bl_{J_{\CF^{i-1}}^{(i)}} B_{\CF^{i-1}} . \]
	Like before, we consider the principal open subset $W_{\CF^i} \defeq \Spec B_{\CF^i} \subset \rho_{i}^{-1} ( W_{\CF^{i-1}} )$ with $B_{\CF^i}$ being the coordinate ring of elements of degree $0$ of the localization of $\Bl_{J_{\CF^{i-1}}^{(i)}} B_{\CF^{i-1}}$ at the element
	\[ h_{V_{n-i-2}\setminus V_{n-i-1}}^{V_{n-i-1}\setminus V_{n-i}} \defeq \;\;\;\;\;\;\;\;\; \prod_{\mathclap{\substack{ v\in V_{n-i-1}\setminus V_{n-i} \\ w\in V_{n-i-2}\setminus V_{n-i-1} }} } \;\;\;\;\;\;\; \frac{v}{w} . \]
	We find that
	\begin{align}\label{BlowupcoordinateringforQEquation}
		B_{\CF^i} = \Big( \Bl_{J_{\CF^{i-1}}^{(i)}} B_{\CF^{i-1}} \Big)_{\left[ h_{V_{n-i-2}\setminus V_{n-i-1}}^{V_{n-i-1}\setminus V_{n-i}} \right]_0 } 
			= k \! \left[ \alggen{V_{n-1}}{V_{n-1}}{\{0\}}, \ldots , \alggen{V_{n-i-1}}{V_{n-i-1}}{V_{n-i}},  \alggen{V_{n-i-2}}{V}{V_{n-i-1}} \right] . 
	\end{align}
	Like before, we also compute the total transform of $H_j$ in $W_{\CF^i}$, for $j\geq i$:
	\begin{align}\label{BlowuptotaltransformforQEquation}
		J_{\CF^{i-1}}^{(j)} B_{\CF^i} = \bigg( \frac{e_{n-i}}{e_{n-i-1}} \bigg) \bigg( \idealgen{V_{n-i-2}}{V}{V_{n-j-1}} \bigg) \subset B_{\CF^i} .
	\end{align}
	In particular for $j=i$, the exceptional divisor of $\rho_i$ in $W_{\CF^i}$ is given by $\big( \frac{e_{n-i}}{e_{n-i-1}} \big) \subset B_{\CF^i}$. Hence the strict transform of $H_j$ in $W_{\CF^i}$, for $j\geq i+1$, is given by 
	\[ J_{\CF^i}^{(j)} \defeq \left( \idealgen{V_{n-i-2}}{V}{V_{n-j-1}} \right) .\]
	Like before, the ideal $\big( \frac{e_{n-i}}{e_{n-i-1}} \big) \subset B_{\CF^i}$ does not depend on the choice of the basis.
	This concludes the construction of $W_{\CF^i} \subset \wt{\bQ}_i$.
	
	The following lemma shows that the sets $W_\CF = W_{\CF^{n-1}}$ cover $\wt{\bQ}_V$ if $\CF$ runs through all complete flags $\CF$ of $V$.
	
\begin{lemma}\label{BlowupcoverLemmaforQ}
	For fixed $i \in \{0, \ldots, n-1\}$, let $W_{\CF^{i-1}} = \Spec B_{\CF^{i-1}} \subset \wt{\mathrm{\mathbf{Q}}}_{i-1}$ be the open subset corresponding to the flag $\CF^{i-1} = \left( \{0\} \subsetneq \ldots \subsetneq V_{n-i-1} \subset V \right) $ with $\dim_k \, V_{n-i-1} = i+1$.
	Then, as $ \CF^{i\prime } = \left( \{0\} \subsetneq \ldots \subsetneq V_{n-i-1} \subsetneq V_{n-i-2}' \subset V \right) $ ranges over all flags with $\dim_k \, V_{n-i-2}' = i+2$, the corresponding open subsets $W_{\CF^{i\prime }} = \Spec B_{\CF^{i\prime }}$ cover $\rho_{i}^{-1} \left( W_{\CF^{i-1}} \right) \subset \wt{\mathrm{\mathbf{Q}}}_{i}$.
	Here, we write
	\begin{align*}
		B_{\CF^{i\prime }} \defeq B_{\CF^{i-1}}  \! \left[ \alggen{V_{n-i-2}'}{V}{V_{n-i-1}} \right] = k \! \left[ \alggen{V_{n-1}}{V_{n-1}}{\{0\}} , \ldots, \alggen{V_{n-i-1}}{V_{n-i-1}}{V_{n-i}} ,  \alggen{V_{n-i-2}'}{V}{V_{n-i-1}} \right] .
	\end{align*}
	Furthermore, the open subsets $W_{\CF^{-1 \prime}} = \Spec B_{\CF^{-1 \prime}}$, defined analogously, each contain $\Omega_V$, and form a covering of $\mathrm{\mathbf{Q}}_V$.
\end{lemma}

	We see that there is an isomorphism between $U_\CF $ and $W_\CF$ that extends the identity on $\Omega_V$, since
	\[A_{\CF_{n-1}} = k \! \left[ \alggen{V}{V}{V_0} , \alggen{V_0}{V_0}{V_1}, \ldots, \alggen{V_{n-2}}{V_{n-2}}{V_{n-1}} , \alggen{V_{n-1}}{V_{n-1}}{\{0\}} \right] = B_{\CF^{n-1}} , \]
	for all complete flags $\CF = \left( \{0\} \subsetneq V_{n-1} \subsetneq \ldots \subsetneq V_0 \subsetneq  V \right)$ of $V$.
	
	Moreover, the claim about the exceptional divisors of $\wt{\bP}_V \cong \wt{\bQ}_V$ is local, hence we consider a fixed open subset $U_\CF \cong W_\CF$, for a complete flag $\CF$ of $V$.
	The only exceptional divisors $D_{V'}$ on $\wt{\bP}_{V}$ which have non-empty intersection with $U_{\CF}$ are the $D_{V'}$ with $V' \in \CF$.
	Analogously, the only exceptional divisors $E_{V'}$ on $\wt{\bQ}_{V}$ which have non-empty intersection with $W_{\CF}$ are the $E_{V'}$ with $V' \in \CF$.
	For $i\in \{0,\ldots,n-1\}$, the divisors $D_{V_i}$ and $E_{V_i}$ are both given by the ideal $\big( \frac{e_{i+1}}{e_i}\big)$ on $U_{\CF} \cong W_{\CF}$, hence they agree.
	This finishes the proof of Theorem \ref{Thm - P Q Isomorphism}.

\begin{proof}[Proof of Lemma \ref{BlowupexplicitLemmaforP}]
	Let $L' \subset \bP_V$ be a linear subspace corresponding to $V' \subset V$ and suppose that $\dim \, L' \geq i$ and $L_i \nsubset L'$, i.e.\ $V' \nsubset V_i$.
	If $L_{i-1} \subset L'$, i.e.\ $V' \subset V_{i-1}$, then, like before, the strict transform of $L'$ in $U_{\CF_{i-1}}$ is given by the ideal
	\[ I_{\CF_{i-1}}' \defeq \left( \idealgen{V'}{V_{i-1}}{V_i}  \right) \subset A_{\CF_{i-1}} . \]
	As there exists some $v\in V'\setminus V_i$, we have	
	\[ 1 = \frac{v}{v} \in I_{\CF_{i-1}}' , \]
	hence the strict transform of $L'$ in $\wt{\bP}_{i-1}$ has empty intersection with $U_{\CF_{i-1}}$.
	If $L_{i-1} \nsubset L'$ then by the statement of this lemma for $U_{\CF_{i-2}}$ the strict transform of $L'$ in $\wt{\bP}_{i-2}$ has empty intersection with $U_{\CF_{i-2}}$.
	From this, the statement for $U_{\CF_{i-1}}$ follows.
\end{proof}
	
\begin{proof}[Proof of Equations \eqref{BlowupcoordinateringforPEquation} and \eqref{BlowuptotaltransformforPEquation}]
	For the equality in \eqref{BlowupcoordinateringforPEquation}, note that
	\[\Big( \Bl_{I_{\CF_{i-1}}^{(i)}} A_{\CF_{i-1}} \Big)_{\left[ g_{V_{i-1}\setminus V_{i}}^{V_{i} \setminus V_{i+1}} \right]_0 } = A_{\CF_{i-1}} \! \left[ \frac{w}{v} \frac{v'}{w'} \middle| \frac{w}{v} \in \alggen{V_{i-1}\setminus V_{i}}{V_i}{V_{i+1}},  \frac{v'}{w'} \in \alggen{V_i}{V_{i-1}}{V_i} \right] . \]
	Thus, we have to consider $\frac{w}{v} \in \alggen{V_{i-1}\setminus V_{i}}{V_i}{V_{i+1}}$ and $\frac{v'}{w'} \in \alggen{V_i}{V_{i-1}}{V_i}$. Then
	\[ \frac{w}{v} \frac{v'}{w'} = \frac{w}{w'} \frac{v'}{v} \in A_{\CF_{i-1}} \! \left[ \alggen{V_i}{V_i}{V_{i+1}} \right] . \]
	For the reverse inclusion, it is enough to consider $\frac{v}{w} \in \alggen{V_i}{V_i}{V_{i+1}}$. We find that
	\[ \frac{v}{w} = \frac{w'}{w} \frac{v}{w'}  \in k \! \left[ \frac{w}{v} \frac{v'}{w'} \middle| \frac{w}{v} \in \alggen{V_{i-1}\setminus V_{i}}{V_i}{V_{i+1}},  \frac{v'}{w'} \in \alggen{V_i}{V_{i-1}}{V_i} \right] , \]
	for arbitrary $w' \in V_{i-1}\setminus V_i$.
	
	Concerning \eqref{BlowuptotaltransformforPEquation}, let $\frac{v}{w} \in \alggen{V_j}{V_{i-1}}{V_i}$ and consider
	\[ \frac{v}{w} = \frac{e_{i+1}}{e_i} \frac{v}{e_{i+1}} \frac{e_i}{w} \in \bigg( \frac{e_{i+1}}{e_i} \bigg) \bigg( \idealgen{V_j}{V_i}{V_{i+1}} \bigg) . \]
	Conversely, let $\frac{v}{w} \in \alggen{V_j}{V_i}{V_{i+1}}$. Then
	\[\frac{e_{i+1}}{e_i} \frac{v}{w} = \frac{v}{e_i} \frac{e_{i+1}}{w} \in  \left( \idealgen{V_j}{V_{i-1}}{V_i} \right)  A_{\CF_i} . \]
\end{proof}
	
\begin{proof}[Proof of Lemma \ref{BlowupcoverLemmaforP}]
	Fix $i \in \{0, \ldots, n-1\}$. It suffices to show that if $\Fp \subset \Bl_{I_{\CF_{i-1}}^{(i)}} A_{\CF_{i-1}}$ is a homogeneous prime ideal and 
	\[ g_{V_{i-1}\setminus V_{i}}^{V_{i} \setminus V_{i+1}'} \defeq \;\;\;\; \prod_{\mathclap{\substack{ v\in V_{i} \setminus V_{i+1}' \\ w\in V_{i-1} \setminus V_{i} } }} \;\;\; \frac{v}{w} \in  \Fp  , \]
	for all $V_{i+1}' \subset V_{i}$ of $\dim_k \, V_{i+1}' = n-i-1$, we already have $I_{\CF_{i-1}}^{(i)} \subset \Fp$, because then $\Fp$ contains the augmentation ideal.
	We define 
	\[ V_{\Fp} \defeq \left\{ v \in V_{i} \middle| \exists w \in V_{i-1}\setminus V_{i} \colon \frac{v}{w} \in \Fp \right\} . \]
	Observe that if $v\in V_{\Fp}$, i.e.\ if there exists $w\in V_{i-1}\setminus V_{i}$ such that $\frac{v}{w} \in \Fp$, then already for all $w' \in V_{i-1}\setminus V_{i}$:
	\[ \frac{v}{w'} = \frac{v}{w} \frac{w}{w'} \in \Fp . \]
	This implies that $V_{\Fp}$ is closed under addition and hence a subspace of $ V_{i}$.
		
	Let $v_1, \ldots, v_l$ be a basis of $V_{\Fp}$ and assume, by the way of contradiction, $l < \dim_k \, V_{i}  =n-i $.
	We can complete $V_{\Fp} $ to a subspace $V_{i+1}' \subset V_{i}$ with $\dim_k \, V_{i+1}' = n-i-1$.
	Then, as $ g_{V_{i-1}\setminus V_{i}}^{V_{i} \setminus V_{i+1}'} \in \Fp$ and $\Fp$ is a prime ideal, there exist $v_{l+1} \in V_{i}\setminus V_{i+1}'$ and $ w \in V_{i-1}\setminus V_{i} $ with $\frac{v_{l+1}}{w} \in \Fp$, i.e.\ $v_{l+1} \in V_{\Fp}$.
	Then $v_1, \ldots, v_l, v_{l+1}$ are linearly independent which is a contradiction.
	In conclusion, we have $V_{\Fp} = V_{i}$ and hence $I_{\CF_{i-1}}^{(i)} \subset \Fp$, using the above observation.
	
	For the second assertion, the same method applies, with the according modifications.
\end{proof}	
	
\begin{proof}[Proof of Lemma \ref{BlowupexplicitLemmaforQ}]
	Let $H' \subset \bQ_V$ be a linear subspace corresponding to $V' \subset V$ and suppose that $\dim \, H' \geq i$ and $H_i \nsubset H'$, i.e.\ $V_{n-i-1} \nsubset V'$.
	If $H_{i-1} \subset H'$, i.e.\ $V_{n-i} \subset V'$, then the strict transform of $H'$ in $W_{\CF^{i-1}}$ corresponds to the ideal
	\[ J_{\CF^{i-1}}' \defeq \left( \idealgen{V_{n-i-1}}{V}{V'}  \right) \subset B_{\CF^{i-1}} . \]
	As there exists some $v\in V_{n-i-1}\setminus V'$, we find that		
	\[ 1 = \frac{v}{v} \in J_{\CF^{i-1}}' . \]
	Therefore the strict transform of $H'$ in $\wt{\bQ}_{i-1}$ has empty intersection with $W_{\CF^{i-1}}$.	
	If $H_{i-1} \nsubset H'$, then the assertion of this lemma for $W_{\CF^{i-2}}$ implies that the strict transform of $H'$ in $\wt{\bQ}_{i-2}$ has empty intersection with $W_{\CF^{i-2}}$. This implies the statement for $W_{\CF^{i-1}}$.
\end{proof}
	
\begin{proof}[Proof of Equation \eqref{BlowupcoordinateringforQEquation} and \eqref{BlowuptotaltransformforQEquation}]
	For equation \eqref{BlowupcoordinateringforQEquation} we have
	\begin{align*}
		\Big( \Bl_{J_{\CF^{i-1}}^{(i)}} B_{\CF^{i-1}} \Big)_{\left[ h_{V_{n-i-2}\setminus V_{n-i-1}}^{V_{n-i-1}\setminus V_{n-i}} \right]_0 } 
			= B_{\CF^{i-1}} \! \left[ \frac{w}{v} \frac{v'}{w'} \middle| \frac{w}{v} \in \alggen{ V_{n-i-2}\setminus V_{n-i-1} }{V_{n-i-1}}{V_{n-i}}, \frac{v'}{w'} \in \alggen{V_{n-i-1}}{V}{V_{n-i-1}} \right] .
	\end{align*}
	First, we consider $\frac{w}{v} \in \alggen{ V_{n-i-2}\setminus V_{n-i-1} }{V_{n-i-1}}{V_{n-i}}$ and $\frac{v'}{w'} \in \alggen{V_{n-i-1}}{V}{V_{n-i-1}}$.
	Then
	\[ \frac{w}{v} \frac{v'}{w'} = \frac{v'}{v} \frac{w}{w'} \in k \! \left[ \alggen{V_{n-i-1}}{V_{n-i-1}}{V_{n-i}} ,  \alggen{V_{n-i-2}}{V}{V_{n-i-1}}  \right] . \]
	We also have to consider elements $\frac{v}{w} \in \alggen{V_{n-i-1}}{V}{V_{n-i}}$.
	If $w \in V_{n-i-1}\setminus V_{n-i}$, then
	\[ \frac{v}{w} \in k \! \left[ \alggen{V_{n-i-1}}{V_{n-i-1}}{V_{n-i}} \right]  , \]
	else $w\in V\setminus V_{n-i-1}$ and we have
	\[ \frac{v}{w} \in k \! \left [ \alggen{V_{n-i-2}}{V}{V_{n-i-1}} \right] .\]	
	For the reverse inclusion, we surely have $\alggen{V_{n-i-1}}{V_{n-i-1}}{V_{n-i}}\subset \alggen{V_{n-i-1}}{V}{V_{n-i}}$.
	It remains to consider $\frac{v}{w} \in \alggen{V_{n-i-2}}{V}{V_{n-i-1}}$.
	If $v\in V_{n-i-2} \setminus V_{n-i-1}$, then we have for arbitrary $w' \in V_{n-i-1} \setminus V_{n-i}$:
	\[\frac{v}{w} = \frac{v}{w'} \frac{w'}{w} \in k \! \left[ \frac{w}{v} \frac{v'}{w'} \middle| \frac{w}{v} \in \alggen{ V_{n-i-2}\setminus V_{n-i-1} }{V_{n-i-1}}{V_{n-i}}, \frac{v'}{w'} \in \alggen{V_{n-i-1}}{V}{V_{n-i-1}} \right] . \]
	For $v\in V_{n-i-1}$ we already have
	\[\frac{v}{w} \in k \! \left[ \alggen{V_{n-i-1}}{V}{V_{n-i}} \right] . \]
	
	For \eqref{BlowuptotaltransformforQEquation}, we consider $\frac{v}{w} \in \alggen{V_{n-i-1}}{V}{V_{n-j-1}}$:
	\[ \frac{v}{w} = \frac{e_{n-i}}{e_{n-i-1}} \frac{e_{n-i-1}}{w} \frac{v}{e_{n-i}} \in \bigg( \frac{e_{n-i}}{e_{n-i-1}} \bigg) \bigg( \idealgen{V_{n-i-2}}{V}{V_{n-j-1}} \bigg) . \]
	Also for $\frac{v}{w} \in \alggen{V_{n-i-2}}{V}{ V_{n-j-1}}$, we have
	\[ \frac{e_{n-i}}{e_{n-i-1}} \frac{v}{w} = \frac{e_{n-i}}{w} \frac{v}{e_{n-i-1}} \in \left( \idealgen{V_{n-i-1}}{V}{V_{n-j-1}} \right)  B_{\CF^{i}} . \]
\end{proof}
	
\begin{proof}[Proof of Lemma \ref{BlowupcoverLemmaforQ}]
	Fix $i \in \{0, \ldots, n-1\}$. It suffices to show that if $\Fq \subset \Bl_{J_{\CF^{i-1}}^{(i)}} B_{\CF^{i-1}}$ is a homogeneous prime ideal such that
	\begin{align*}
		h_{V_{n-i-2}' \setminus V_{n-i-1}}^{V_{n-i-1} \setminus V_{n-i}} \defeq \;\;\;\;\;\;\;\;\; \prod_{\mathclap{\substack{ v\in V_{n-i-1} \setminus V_{n-i} \\ w \in V_{n-i-2}' \setminus V_{n-i-1} }}} \;\;\;\;\;\;\; \frac{v}{w} \in \Fq ,
	\end{align*}
	for all subspaces $V_{n-i-2}' \supset V_{n-i-1}$ with $\dim_k \, V_{n-i-2}' = i+2$, then $J_{\CF^{i-1}}^{(i)} \subset \Fq$, since then $\Fq$ already contains the augmentation ideal.
	For such a prime ideal $\Fq$, we define the set
	\[ V_\Fq \defeq \{0\} \cup \left\{ w\in V\setminus \{0\} \middle| \exists v\in V_{n-i-1} \colon \frac{v}{w} \in B_{\CF^{i-1}} \land \frac{v}{w} \notin \Fq \right\} . \]
	Note that $V_{n-j-1}\setminus V_{n-j} \subset V_\Fq $, for $j=0,\ldots, i$, because we have $ \frac{w}{w} = 1 \notin \Fq$ for such $w \in V_{n-j-1}\setminus V_{n-j}$.
	Hence we conclude that $V_{n-i-1} \subset V_\Fq$.
	Now assume that there exists some $\tilde{w} \in V_\Fq \setminus V_{n-i-1}$.
	Then the subspace $V_{n-i-2}' \defeq V_{n-i-1} \oplus k \tilde{w}$ has $\dim_k \, V_{n-i-2}' = i+2$.
	As 
	\[ h_{V_{n-i-2}' \setminus V_{n-i-1}}^{V_{n-i-1} \setminus V_{n-i}} \in \Fq , \]
	there exist $v \in V_{n-i-1}\setminus V_{n-i}$ and $w\in V_{n-i-2}' \setminus V_{n-i-1}$ such that $\frac{v}{w} \in \Fq$.
	Now write $w = w' + \lambda \tilde{w}$, for $w' \in V_{n-i-1}$ and $\lambda \neq 0$.
	If we have $w' = 0$, then $ \frac{v}{\tilde{w}} \in \Fq$.
	Otherwise, we compute that
	\[ \frac{v}{\lambda \tilde{w}} =  \left( 1 + \frac{w'}{\lambda \tilde{w} } \right) \frac{v}{w} \in \Fq , \]
	hence $\frac{v}{\tilde{w}} \in \Fq$ in this case, too.
	Let $v' \in V_{n-i-1}$ be arbitrary.
	We find that
	\[\frac{v'}{\tilde{w}} = \frac{v'}{v} \frac{v}{\tilde{w}} \in \Fq . \]
	But because $v' \in V_{n-i-1}$ is arbitrary, this contradicts $\tilde{w} \in V_\Fq$.
	This implies that $V_\Fq = V_{n-i-1}$.
	For 
	\[\frac{v}{w} \in \alggen{V_{n-i-1}}{V}{V_{n-i-1}}= \alggen{V_{n-i-1}}{V}{V_\Fq} , \]
	it follows by the definition of $V_\Fq$ that $\frac{v}{w}\in \Fq$.
	Hence we have $J_{\CF^{i-1}}^{(i)} \subset \Fq$.
	
	The analogue of this method also proves the second assertion.
\end{proof}

\subsection{Proof of Proposition \ref{Prop - B Incidence variety X_i isomorphism}}\label{Subsect - Proof of Prop B Incidence variety X_i isomorphism}

	We use the same strategy as in the proof of Theorem \ref{Thm - P Q Isomorphism}:
	We construct an open affine covering $\CU_{\CF_i}$ of $\RX_i$, indexed by truncated flags $\CF_i$, so that there are isomorphisms $U_{\CF_i} \myrightarrow{\sim} \CU_{\CF_i}$ (where $U_{\CF_i}$ is the open covering of $\wt{\bP}_i$ constructed in the proof of Theorem \ref{Thm - P Q Isomorphism}).
	All of these isomorphisms will agree on $\Omega_V$ and its image in $\RX_{i}$, therefore they glue to 
	\[b_i \colon \wt{\bP}_{i} \myrightarrow{\sim} \RX_i ,\]
	as $\wt{\bP}_{i}$ and $\RX_i$ are reduced and separated over $k$.

	We begin by defining these open subsets of $\RX_i$.
	For any $W' \subset W \subset V$, let $U_{W'} \subset \bP_{W}$ denote the open affine subset
	\[ U_{W'} \defeq \Spec k \! \left[ \frac{v}{w} \middle| v\in W, w\in W\setminus W' \right] \subset \bP_W . \]
	Consider for all flags $\CF_i = \left( \{0\}\subset V_{i+1} \subsetneq\ldots \subsetneq V_0 \subsetneq V_{-1} = V \right) $, with $\dim_k \, V_{j} = n-j$, the open subsets
	\[ \CU_{\ul{\CF_i}} \defeq    U_{V_0}\times \ldots\times U_{V_{i+1}} \times \;\; \prod\limits_{\mathclap{\substack{W \subset V, W \notin \{ V,\ldots,V_i \}, \\ \dim_k \, W \geq n-i  }}} \; \bP_{W} \quad  \subset \bP_{\ul{V}, i} \]
	which form a covering of $\bP_{\ul{V}, i}$.
	Each of these open subsets comes with a natural projection
	\begin{equation}\label{Eq - Projection U_F}
	\begin{aligned}
		\bP_{\ul{V}, i}	&\lra	\bP_{V} \times\ldots\times \bP_{V_i}	\\
		\cup \hspace{6pt} & \hspace{51.7pt}	\cup								\\
		\CU_{\ul{\CF_i}}	&\lra U_{V_0}\times \ldots\times U_{V_{i+1}}. 
	\end{aligned}
	\end{equation}
	Note that we have $U_{V_0}\times \ldots\times U_{V_{i+1}} = \Spec \CA_{\CF_i}$, where 
	\[\CA_{\CF_i} \defeq k \! \left[ \alggen{V}{V}{V_0} \right] \otimes \ldots \otimes k \! \left[ \alggen{V_i}{V_i}{V_{i+1}} \right] .\]
	By writing the tensor product here, we want to emphasize that in this context there are no relations between elements of $\alggen{V_j}{V_{j}}{V_{j+1}}$ and $\alggen{V_{j'}}{V_{j'}}{V_{j'+1}}$, for $j \neq j'$.
	Then
	\[ \CU_{\ul{\CF_i}} = \mathrm{multiProj} \; \CA_{\ul{\CF_i}} \]
	where $\CA_{\ul{\CF_i}}$ is the multi-graded $\CA_{\CF_i}$-algebra
	\[ \CA_{\ul{\CF_i}} \defeq \CA_{\CF_i} \otimes \;\; \bigotimes\limits_{\mathclap{\substack{W \subset V, W \notin \{ V,\ldots,V_i \}, \\ \dim_k \, W \geq n-i  }}} \; k \! \left[ v \middle| v\in W \right] \]
	with a grading for each $W \subset V, W \notin \{ V,\ldots,V_i \}$, with $\dim_k \, W \geq n-i$.
	Define 
	\[ \CU_{\CF_i} \defeq \RX_i \cap \CU_{\ul{\CF_i}} \]
	so that 
	\[ \CU_{\CF_i} = \mathrm{multiProj} \;  \CA_{\ul{\CF_i}} / J_{\ul{\CF_i}}  \]
	where $J_{\ul{\CF_i}}$ is the incidence ideal induced form $\CJ_i$.
	Furthermore, let
	\[ \Omega_{\RX_i} \defeq \RX_i \cap \bigg( \;\; \prod\limits_{\mathclap{\substack{W \subset V, \\ \dim_k \, W \geq n-i  }}} \; \Omega_{W} \bigg) \subset \CU_{\CF_i} \]
	which is a dense open subset of $\RX_i$. The coordinate ring of $ \Omega_{\RX_i}$ is given by
	\[ \CO_{\RX_i} (\Omega_{\RX_i})  =  \bigg( \;\; \bigotimes\limits_{\mathclap{\substack{W \subset V, \\ \dim_k \, W \geq n-i  }}} \; k \! \left[ \frac{v}{w} \middle| v\in W, w \in W\setminus \{0\} \right] \bigg) \bigg/ J_{\Omega_{\RX_i}}  \]
	where $J_{\Omega_{\RX_i}}$ is the induced incidence ideal.	
	
	Now let $\RX_{\CF_i}$ be the incidence variety
	\[\RX_{\CF_i} \subset \bP_V \times\ldots\times \bP_{V_i}  \]
	so that we have
	\[ \RX_{\CF_i} \cap ( U_{V_0} \times\ldots\times U_{V_{i+1}} ) = \Spec \CA_{\CF_i} / J_{\CF_i} \]
	where $J_{\CF_i}$ is the corresponding incidence ideal.
	Moreover, define the open affine subset
	\[\Omega_{\CF_i} \defeq \RX_{\CF_i} \cap \left( \Omega_{V} \times\ldots\times \Omega_{V_i} \right) \subset \RX_{\CF_i}  \]	
	with coordinate ring
	\[ \CO_{\RX_{\CF_i}} ( \Omega_{\CF_i} ) = \bigslant{ \bigg( k \! \left[ \alggen{V}{V}{\{0\}} \right] \otimes\ldots\otimes k \! \left[ \alggen{V_i}{V_i}{\{0\}} \right] \bigg) }{J_{\Omega_{\CF_i}}}  \]
	where $J_{\Omega_{\CF_i}}$ is the induced incidence ideal.
	
	To construct the isomorphism $b_i \colon \wt{\bP}_i \myrightarrow{\sim} \RX_i$, it suffices to prove the subsequent claim:
	There are the following isomorphisms in the commutative diagram
	\begin{equation}\label{Eq - X_F_i big commutative diagram}
		\centering
		\begin{tikzcd}
			\wt{\bP}_i												&		& \RX_i \\
			U_{\CF_i}	\arrow[r,"\sim"] \arrow[u,hook]	&\RX_{\CF_i} \cap \left( U_{V_0} \times\ldots\times U_{V_{i+1}} \right) \arrow[r,leftarrow,"\sim"]	&\CU_{\CF_i} \arrow[u,hook] \\
			\Omega_V\arrow[r,"\sim"] \arrow[u,hook]	&\Omega_{\CF_i} \arrow[r,leftarrow,"\sim"] \arrow[u,hook]	&\Omega_{\RX_i} \mathclap{\; ,} \arrow[u,hook]	
		\end{tikzcd} 
	\end{equation}
	and the bottom morphisms are given by the following $k$-algebra isomorphisms of the coordinate rings
	\begin{alignat*}{3}
		\CO_{\wt{\bP}_i} (\Omega_V)	&\overset{\sim}{\longleftarrow}	\CO_{\RX_{\CF_i}}(\Omega_{\CF_i})	&&\myrightarrow{\sim}	\CO_{\RX_i}(\Omega_{\RX_i})		\\
		\frac{v}{w} &\longmapsfrom	\frac{v}{w} \qquad	\;\;\,		\frac{v}{w}	&&\longmapsto		\frac{v}{w} \mathclap{\;.}	
	\end{alignat*}
	Furthermore, we will see from the construction of these isomorphisms in \eqref{Eq - X_F_i big commutative diagram} that they are compatible with $\pi_i \colon \wt{\bP}_i \lra \wt{\bP}_{i-1}$, and the projection $\RX_i \lra \RX_{i-1}$, as well as the projection $\RX_{\CF_i}\cap (U_{V_0} \times\ldots\times U_{V_{i+1}}) \lra \RX_{\CF_{i-1}} \cap ( U_{V_0} \times\ldots\times U_{V_i})$, respectively.

	For the right side of \eqref{Eq - X_F_i big commutative diagram}, we use the following lemma which we will prove at the end of this section.

\begin{lemma}\label{Lemma - X_i open subset reduction isomorphism}
	The projection
	\begin{align}\label{Eq - Projection X_i on U_F}
		\CU_{\CF_i} \lra \RX_{\CF_i} \cap \left( U_{V_0} \times\ldots\times U_{V_{i+1}} \right) ,
	\end{align}
	induced by \eqref{Eq - Projection U_F}, is an isomorphism.
	It induces an isomorphism
	\[ \Omega_{\RX_i} \myrightarrow{\sim} \Omega_{\CF_i} \]
	which is given as a $k$-algebra isomorphism of the coordinate rings
	\begin{align*}
		\CO_{\RX_{\CF_i}}(\Omega_{\CF_i})	&\myrightarrow{\sim}	\CO_{\RX_i}(\Omega_{\RX_i})	\\
		\frac{v}{w} &\longmapsto \frac{v}{w} .
	\end{align*}	
\end{lemma}

	For the left side of \eqref{Eq - X_F_i big commutative diagram}, by using induction, we may suppose that there is an isomorphism
	\begin{equation}\label{Eq - Commutative Diagram U_F_i and X_F_i}
		\centering
		\begin{tikzcd}
			U_{\CF_{i}} \arrow[d] 		& \RX_{\CF_{i}} \cap (U_{V_0} \times\ldots\times U_{V_{i+1}} ) \arrow[d] \\
			U_{\CF_{i-1}} \arrow[r, "\sim"] & \RX_{\CF_{i-1}} \cap (U_{V_0} \times\ldots\times U_{V_i} )
		\end{tikzcd}
	\end{equation}
	which takes on the level of $k$-algebras the form
	\begin{equation*}
		\centering
		\begin{tikzcd}
			A_{\CF_{i}}  		& \CA_{\CF_{i}} / J_{\CF_{i}}  \\
			A_{\CF_{i-1}} \arrow[u,hook] \arrow[r, leftarrow, "\sim"] &  \CA_{\CF_{i-1}} / J_{\CF_{i-1}} \arrow[u,hook] \mathclap{.}
		\end{tikzcd}
	\end{equation*}
	Then we find that
	\[\CA_{\CF_{i}} / J_{\CF_{i}} \cong \bigslant{\left( A_{\CF_{i-1}} \otimes k \! \left[ \alggen{V_i}{V_i}{V_{i+1}} \right] \right)}{J_{\CF_i}'} \]
	where
	\[ J_{\CF_i}' = \left( \frac{v}{w}\otimes\frac{v'}{w'} - \frac{v'}{w}\otimes\frac{v}{w'} \middle| \frac{v}{w} \in \alggen{V}{V}{V_0}\cup\ldots\cup\alggen{V_{i-1}}{V_{i-1}}{V_i} \wedge v \in V_{i}, \frac{v'}{w'} \in \alggen{V_i}{V_i}{V_{i+1}}\right) .\]
	Analogous to the blow-up of the origin in $\BA^n$, the following lemma holds (use for example \cite[Ch.\ 5, Satz 5.10]{K80}):

\begin{lemma}\label{Lemma - A_F Algebra isomorphism}
	The $A_{\CF_{i-1}}$-algebra homomorphism
	\begin{align*}
		\CA_{\CF_i}/J_{\CF_i} \cong \bigslant{\left( A_{\CF_{i-1}}\otimes k \! \left[ \alggen{V_i}{V_i}{V_{i+1}}\right] \right) }{J_{\CF_i}' }  &\lra A_{\CF_i} \; \\
		\frac{v}{w} \otimes 1 &\longmapsto \frac{v}{w}	\\
		1 \otimes \frac{v}{w} &\longmapsto \frac{v}{w}
	\end{align*}
	is an isomorphism.
\end{lemma}

	This gives an isomorphism
	\[ U_{\CF_i} \myrightarrow{\sim} \RX_{\CF_{i}} \cap (U_{V_0} \times\ldots\times U_{V_{i+1}} ) \]
	such that \eqref{Eq - Commutative Diagram U_F_i and X_F_i} commutes.
	By further localizing we find that 
	\[ \Omega_V \myrightarrow{\sim} \Omega_{\CF_i} \]
	is given by the $k$-algebra isomorphism 
	\begin{align*}
		\CO_{\RX_{\CF_i}} (\Omega_{\CF_i} ) &\myrightarrow{\sim} \CO_{\wt{\bP}_i} (\Omega_V) = k \! \left[ \frac{v}{w} \middle| v \in V, w \in V \setminus\{0\} \right] \\
		\frac{v}{w} &\longmapsto \frac{v}{w} .
	\end{align*}

	For the second claim we show that
	\begin{align}\label{Eq - Divisor in X_i}
		\RX_i \cap \Bigg(  \;\;\;\;  \prod\limits_{\mathclap{\substack{V' \subsetneq W \subset V}}} \; \bP_{W/V'} \times \bP_{V'} \times \;\; \prod\limits_{\mathclap{\substack{V' \nsubset W \subset V,  \\ \dim_k \, W \geq n-i}}} \; \bP_{W} \Bigg) 
	\end{align} 
	and the divisor $b_i ( D_{V'} )$ agree on each $\CU_{\CF_i}$.
	For a truncated flag $\CF_i$, let $j'$ be the smallest $j'\in \{-1,\ldots,i-1\}$ such that $V' \subsetneq V_{j'}$.
	Then on $\CU_{\CF_i}$, \eqref{Eq - Divisor in X_i} is given by the ideal
	\begin{align*}
		I_{\CF_i} \defeq \sum\limits_{j= -1}^{j'} \left( \alggen{V'}{V_j}{V_{j+1}} \right) \subset \CA_{\CF_i} / J_{\CF_i}.
	\end{align*}
	If $V' \neq V_{j'+1}$, then there exists $v \in V' \setminus V_{j'+1}$, hence
	\[1 = \frac{v}{v} \in \left( \alggen{V'}{V_{j'}}{V_{j'+1}} \right). \]
	Therefore, it is $I_{\CF_i} =\CA_{\CF_i} / J_{\CF_i}$, i.e.\ \eqref{Eq - Divisor in X_i} does not intersect $\CU_{\CF_i}$.
	Moreover, in this case $D_{V'}$ has empty intersection with $U_{\CF_i}$, by the analogous argument in $A_{\CF_i}$, so that the claim holds.
	
	On the other hand, if $V' =  V_{j'+1}$ one computes that 
	\begin{align}
		I_{\CF_i} = \bigg( \frac{e_{j'+2}}{e_{j'+1}} \bigg) 
	\end{align}
	where, like before, $e_0,\ldots,e_n$ is a basis of $V$ which is subordinate to $\CF_i$.
	But this is exactly the ideal of $b_i ( D_{V'} )$ in $\CU_{\CF_i}$, hence the claim also follows in this case.

\begin{proof}[Proof of Lemma \ref{Lemma - X_i open subset reduction isomorphism}]
	It suffices to show the statement so to say for one subspace at a time:
	Let $V' \subset V$, $V' \notin \{V,\ldots,V_i\}$ be a subspace with $\dim_k \, V' \geq n-i$, and let
	\[ S \subset \left\{ W \subset V \middle| \dim_k \, W \geq n-i \right\}  \text{, such that $\{V', V\ldots, V_i\} \subset S$.} \]
	Let $S' \defeq S \setminus \{V', V\ldots, V_i\}$.
	Then we claim that the projection
	\begin{align}\label{Eq - X'_i X''_i morphism}
		\RX'_i \cap \bigg( U_{V_0} \times\ldots\times U_{V_{i+1}} \times \bP_{V'} \times \;\; \prod\limits_{\mathclap{\substack{W \in S' }}} \; \bP_{W} \bigg) 
			\lra \RX''_i \cap \bigg( U_{V_0} \times\ldots\times U_{V_{i+1}} \times \;\; \prod\limits_{\mathclap{\substack{W \in S' }}} \; \bP_{W} \bigg) 
	\end{align}
	where $\RX'_i$ and $\RX''_i$ are the respective incidence varieties, is in fact an isomorphism.

	Let 
	\begin{align*}
		\CA' &= \CA_{\CF_i} \otimes k \! \left[ v \middle| v \in V' \right] \otimes \;\; \bigotimes\limits_{\mathclap{\substack{W \in S' }}} \; k \! \left[ v \middle| v \in W \right] , 	\\
		\CA''&= \CA_{\CF_i}  \otimes \;\; \bigotimes\limits_{\mathclap{\substack{W \in S' }}} \; k \! \left[ v \middle| v \in W \right] 
	\end{align*}
	so that \eqref{Eq - X'_i X''_i morphism} is the morphism
	\begin{align}\label{Eq - X'_i X''_i morphism multiProj}
		\mathrm{multiProj} \; \CA' / J' \lra 	\mathrm{multiProj} \; \CA'' / J''
	\end{align}
	where $J'$ and $J''$ are the incidence ideals of $\RX'_i$ respectively $\RX''_i$.
	Let $j' \in \{-1,\ldots,i-1\}$ such that $V'\subset V_{j'}$ and $V' \nsubset V_{j'+1} $, and choose an arbitrary $\tilde{v} \in V' \setminus V_{j'+1}$.
	Using the relations of the incidence ideal, one computes that
	\begin{align}\label{Eq - J' J'' ideal equality}
		J' = J'' \CA' + \left( 1 \otimes v - \frac{v}{\tilde{v}} \otimes \tilde{v} \middle| v \in V' \right) \subset  \CA' .
	\end{align}
	Note that $\CA'$ is a graded $\CA''$-algebra with the grading $\deg v = 1 $ for all $v \in V'$.
	If we consider the quotients, it follows from \eqref{Eq - J' J'' ideal equality} that $\CA' / J' $ is in fact generated as a $\CA'' / J''$-algebra by the single element $\tilde{v} \in V'$ with $\deg \tilde{v} = 1$:
	\[ \CA' /J' = (\CA'' / J'') [ \tilde{v} ] .\]
	This implies that \eqref{Eq - X'_i X''_i morphism multiProj} is an isomorphism.
	Moreover, let
	\begin{align*}
		\Omega_{\RX'_i} &\defeq \RX'_i \cap \bigg( \Omega_V \times\ldots\times \Omega_{V_{i}} \times \Omega_{V'} \times \;\; \prod\limits_{\mathclap{\substack{W \in S' }}} \; \Omega_{W} \bigg)	, \\
		\Omega_{\RX''_i}&\defeq \RX''_i \cap \bigg( \Omega_V \times\ldots\times \Omega_{V_{i}} \times \;\; \prod\limits_{\mathclap{\substack{W \in S' }}} \; \Omega_{W} \bigg)	.
	\end{align*}
	Then it follows by localizing $\CA'' / J'' \hookrightarrow \CA' / J'$ and taking elements of degree $0$ that
	\begin{align*}
		\Omega_{\RX'_i}	\myrightarrow{\sim} \Omega_{\RX''_i}
	\end{align*}
	is given by 
	\begin{align*}
		\CO_{\RX''_i} (\Omega_{\RX''_i}) &\myrightarrow{\sim} \CO_{\RX'_i} (\Omega_{\RX'_i})	\\
		\frac{v}{w} &\longmapsto \frac{v}{w}	.
	\end{align*}
\end{proof}

\section{Modular interpretation}\label{Sect - Modular interpretation}

\subsection{The functor of points of $\bP_V$, $\bQ_V$ and $\bB_V$}

Recall the well-known description of the functor of points of the projective space $\bP_V$ \cite[Ch.\ II Thm.\ ~7.1]{H77}: 

\begin{proposition}\label{Prop - Points of P}
	The $k$-scheme $\mathrm{\mathbf{P}}_V$ represents the functor on $k$-schemes which associates to a $k$-scheme $T$ the set of isomorphism classes of invertible sheaves $\CL$ on $T$ together with a surjection $ l \colon V \otimes_k \CO_T \myrightarrowdbl{} \CL$.
\end{proposition}

Here, two pairs $(\CL, l)$ and $(\CL', l')$ are \textit{isomorphic} if $\CL \cong \CL'$ are isomorphic as invertible sheaves such that $l$ and $l'$ are compatible with this isomorphism.

Pink and Schieder gave an analogous description for the functor of points of $\bQ_V$ \cite[Ch.\ ~7]{PS14}, \cite[Cor.\ 6.4]{S09}:

\begin{definition}\label{Def - Reciprocal map}
	Let $T$ be a $k$-scheme and $\CL$ an invertible sheaf on $T$. A map 
	\[r \colon V\setminus \{0\} \lra \Gamma(T, \CL)\]
	is called \textit{reciprocal} if 
	\begin{altitemize}
		\item[(i) ] $r(\lambda v ) = \lambda^{-1} r(v)$ for all $v\in V\setminus \{0\}$, $\lambda \in k\unts$; and
		\item[(ii) ] $r(v)  r(v') = r(v+v') \left( r(v)+r(v') \right)$ for all $v,v' \in V\setminus\{0\}$ with $v+v' \neq 0$.
	\end{altitemize}
\end{definition}

\begin{proposition}\label{Prop - Points of Q}
	The $k$-scheme $\mathrm{\mathbf{Q}}_V$ represents the functor on $k$-schemes which associates to a $k$-scheme $T$ the set of isomorphism classes of invertible sheaves $\CL$ on $T$ together with a reciprocal map $ r \colon V\setminus \{0\} \lra \Gamma\left( T, \CL \right)$ whose image are generating sections of $\CL$.
\end{proposition}

Again, two such pairs $(\CL, r)$ and $(\CL', r')$ are \textit{isomorphic} if there is an isomorphism $\CL \cong \CL'$ of invertible sheaves which is compatible with $r$ and $r'$.

For $\Omega_V \subset \bP_V$ and $\Omega_V \subset \bQ_V$, there is the following description \cite[Prop.\ 7.9, 7.11]{PS14}:

\begin{proposition}
	The open subscheme $\Omega_V \subset \mathrm{\mathbf{P}}_V$ represents the subfunctor on $k$-schemes which associates to a $k$-scheme $T$ the set of isomorphism classes $(\CL, l )$ such that all sections $l(v\otimes 1) \in \Gamma(T,\CL)$ are nowhere vanishing, for $v \in V\setminus \{0\}$.
	
	The open subscheme $\Omega_V \subset \mathrm{\mathbf{Q}}_V$ represents the subfunctor on $k$-schemes which associates to a $k$-scheme $T$ the set of isomorphism classes $(\CL, r )$ such that all sections $r(v) \in \Gamma (T, \CL)$ are nowhere vanishing, for $v \in V \setminus \{0\}$.
\end{proposition}

For $(\CL, l) \in \Omega_V (T)$, respectively $(\CL, r) \in \Omega_V (T)$, any nowhere vanishing global section of $\CL$ yields a trivialization $\CL \cong \CO_T$.
Then the birational equivalence of $\bP_V$ and $\bQ_V$ is given in terms of $T$-valued points by:
\begin{alignat*}{2}
	\bP_V (T) \supset \Omega_V (T)	&		\,\, \;\;\!\!\! =			&& \;	\Omega_V (T)	\subset \bQ_V (T)			\\
					(\CO_T,l)		&\longmapsto			&& \, (\CO_T, \tfrac{1}{l})				 .
\end{alignat*}

To construct a desingularization of $\bQ_V$ which also dominates $\bP_V$, Pink and Schieder started by defining another functor on $k$-schemes, called $B_V$ \cite[Ch.\ 10]{PS14}, \cite[Ch.\ ~6]{S09}:
To a $k$-scheme $T$, $B_V$ associates the set of isomorphism classes of objects consisting of the following data:
For every subspace $\{0\} \neq W \subset V$ an invertible sheaf $\CL_{W}$ on $T$, and a surjection
\[l_{W} \colon W \otimes_k \CO_T \myrightarrowdbl{} \CL_{W} .\]
Furthermore, for every inclusion $\{0\} \neq W' \subset W \subset V$ a morphism
\[ \psi_{W'}^{W} \colon \CL_{W'} \lra \CL_{W} \]
of invertible sheaves such that the diagram
\[\begin{tikzcd}
	W \otimes_k \CO_T \arrow[r, two heads, "l_{W}"] 		& \CL_{W} \arrow[d,leftarrow, "\psi_{W'}^{W}"] \\
	W' \otimes_k \CO_T \arrow[r, two heads, "l_{W'}"] \arrow[u, hook]	& \CL_{W'}  
\end{tikzcd}\]
commutes. We denote such an object by $(\CL,l,\psi)$.
Two objects $(\CL,l,\psi)$ and $(\CL',l',\psi')$ are \textit{isomorphic} if for all subspaces $\{0\} \neq W \subset V$ there are isomorphisms $\CL_{W} \cong \CL_{W}^{\prime}$ of invertible sheaves which are compatible with $l_{W}$ and $l_{W}^{\prime}$.
Note that this already implies that $\psi_{W'}^{W}$ and ${\psi'}_{W'}^{W}$ are compatible as the $l_{W}$ are surjective.
To a morphism $f \colon T' \lra T$ of $k$-schemes we associate the map that pulls all this data back along $f$.

Pink and Schieder showed that $B_V$ is represented by a closed subscheme of 
\[\bP_{\ul{V},n-1} = \;\; \prod\limits_{\mathclap{\substack{\{0\} \neq W \subset V}}} \;\;\; \bP_{W} \]
\cite[Thm.\ ~10.16]{PS14}, \cite[Prop.\ ~6.5]{S09}.
Using the following proposition, we observe that this closed subscheme is in fact the incidence variety $\mathrm{X}_{n-1} \subset \bP_{\ul{V},n-1}$.
In particular, the desingularization by Pink and Schieder is isomorphic to $\bB_V$.

\begin{proposition}\label{Prop - Points of B}
	The $k$-scheme $\mathrm{\mathbf{B}}_V \cong \mathrm{X}_{n-1} \subset \mathrm{\mathbf{P}}_{\ul{V},n-1}$ represents the functor $ B_V$ on $k$-schemes, defined as above.
\end{proposition}

\begin{proof}
	Observe that the $k$-scheme $\bP_{\ul{V},n-1}$ represents the functor which associates to a\linebreak $k$-scheme $T$ the set of isomorphism classes of the following data:
	For every subspace\linebreak ${\{0\} \neq W \subset V}$ an invertible sheaf $\CL_{W}$ on $T$ and a surjection $l_{W} \colon  \CO_T \otimes_k W \myrightarrowdbl{} \CL_{W}$.
	More precisely, by the universal property of the fibre product, a morphism $h \colon T \lra \bP_{\ul{V},i}$ corresponds to a family of morphisms $f_{W} \colon T \lra \bP_{W}$ indexed by all $\{0\} \neq W \subset V$.
	These $f_{W}$ in turn correspond to surjections
	\[ l_{W} \colon W \otimes_k  \CO_T  \myrightarrowdbl{} \CL_{W}\]
	such that
	\[ \CL_{W} \cong f_{W}^{\ast} \CO_{\bP_{W}} (1) \, \text{, and }\, f_{W}^{\ast} (v) = l_W (v) \defeq l_{W}(v\otimes 1) \, \text{, for all $v\in W$.} \]

	Such a morphism $h$ factors through $\RX_{n-1}$ if and only if the inverse image ideal sheaf of $\CJ_{n-1}$ under $h$ vanishes.
	By \eqref{Eq - Generators of incidence ideal sheaf}, $h^{-1}\CJ_{n-1}  \cdot \CO_T$ is generated by
	\begin{align}\label{Eq - Inverse image ideal sheaf generators}
		h^{-1}\left( v \otimes v' - v' \otimes v \middle| v,v' \in W' \right)^{\sim} \cdot \CO_T ,
	\end{align}
	for all $\{0\} \neq W' \subset W \subset  V$, where we extend
	\[\left( v \otimes v' - v' \otimes v \middle| v,v' \in W' \right)^{\sim} \subset \pr_W^{\ast} \CO_{\bP_W} \otimes_k  \pr_{W'}^{\ast} \CO_{\bP_{W'}}\]
	in $\CO_{\bP_{\ul{V},n-1}}$.
	Therefore it suffices to show the following Lemma:
\end{proof}

\begin{lemma}\label{Lemma - Special case B points equivalence}
	Fix $\{0\} \neq W' \subset W \subset V$. Then \eqref{Eq - Inverse image ideal sheaf generators} vanishes if and only if there exists a morphism $\psi_{W}^{W'} \colon \CL_{W'} \lra \CL_{W}$ of $\CO_T$-modules such that the following diagram is commutative:
	\[\begin{tikzcd}
		W \otimes_k \CO_T \arrow[r, two heads, "l_{W}"] 		& \CL_{W} \arrow[d,leftarrow, "\psi_{W'}^{W}"] \\
		W' \otimes_k \CO_T \arrow[r, two heads, "l_{W'}"] \arrow[u, hook]	& \CL_{W'}  
	\end{tikzcd}\]
\end{lemma}

\begin{proof}[Proof of Lemma \ref{Lemma - Special case B points equivalence}]
	It is enough to consider the case of $\{0\} \neq V'\subset V$ with
	\begin{equation}\label{Eq - Diagram for psi in special case}
		\centering
		\begin{tikzcd}
			V \otimes \CO_T \arrow[r, two heads, "l"] 		& \CL \arrow[d,leftarrow, "\psi"] \\
			V' \otimes \CO_T \arrow[r, two heads, "l'"] \arrow[u, hook]	& \CL'
		\end{tikzcd}
	\end{equation}
	corresponding to $f \colon T \lra \bP_{V}$ and $f' \colon T \lra \bP_{V'}$. The general case proceeds analogously.
	
	Let, for $w\in V$,
	\[ U_w \defeq \left\{ P \in T \middle| l(w)_P \notin \Fm_P  \CL_P \right\} , \]
	so that $\CL$ is generated by $l(w)$ on $U_w$; and let, for $w' \in V'$,
	\[ U_{w'}' \defeq \left\{ P \in T \middle| l'(w')_P \notin \Fm_P  \CL'_P \right\} \]
	so that $\CL'$ is generated by $l'(w')$ on $U_{w'}'$.
	Then the sets $U_w\cap U_{w'}'$ form a covering of $T$ if $w$ and $w'$ range over all pairs $(w,w') \in V \times V'$.
	Furthermore, on $U_w\cap U_{w'}'$ the inverse image ideal sheaf \eqref{Eq - Inverse image ideal sheaf generators} is generated by the elements
	\begin{align*}
		h^{\ast}\left( \frac{v}{w}\otimes \frac{v'}{w'} - \frac{v'}{w}\otimes \frac{v}{w'} \right) &=
				f^{\ast}\bigg( \frac{v}{w} \bigg) f^{\prime \ast} \bigg( \frac{v'}{w'} \bigg) - f^{\ast}\bigg( \frac{v'}{w} \bigg) f^{\prime \ast} \bigg( \frac{v}{w'} \bigg)	\\
			&= \frac{l(v)}{l(w)}\frac{l'(v')}{l'(w')} - \frac{l(v')}{l(w)}\frac{l'(v)}{l'(w')} ,
	\end{align*}
	for all $v,v' \in V'$.
	
	First, assume that in \eqref{Eq - Diagram for psi in special case} such $\psi$ exists.
	For arbitrary $v, v' \in V'$, let
	\begin{alignat*}{4}
		l(v)  &= 	 a l(w)  ,&&	\qquad			  l(v') \, &=& \, b l(w) ,						\\
		l'(v) &=  a' l'(w')  ,&&	\qquad			  l'(v')\, &=& \, b' l'(w')			
	\end{alignat*}
	with $a,a', b,b' \in \CO_T\res{U_w\cap U_{w'}'}$. By the commutativity of \eqref{Eq - Diagram for psi in special case} we have
	\begin{alignat*}{3}
		l(v)  &=	 \psi \left(  l'(v) \right)				&&=&	\,	a' \psi\left(l'(w')\right)	&,		 \\
		l(v') &= 	 \psi \left(  l'(v') \right)			&&=& \,	b' \psi\left(l'(w')\right) 	&.
	\end{alignat*}
	Hence we compute that
	\begin{align*}
		h^{\ast}\left( \frac{v}{w}\otimes \frac{v'}{w'} - \frac{v'}{w}\otimes \frac{v}{w'} \right) &=
			 \frac{l(v)}{l(w)} \frac{l'(v')}{l'(w')} - \frac{l(v')}{l(w)} \frac{l'(v)}{l'(w')} 				\\
			&= \frac{a' \psi\left(l'(w')\right) }{l(w)} \frac{b' l'(w')}{l'(w')} - \frac{b' \psi\left(l'(w')\right) }{l(w)} \frac{a' l'(w')}{l'(w')} \\
			&= 0 .
	\end{align*}
	
	Conversely, suppose now that $h^{-1} \CJ_{n-1} \cdot \CO_T = 0$. For $w\in V'$, we want to define
	\[\psi\res{U_{w}'} \left( l'(w)\right) \defeq l(w)\]
	and have to check that this glues to give a map $\psi$.
	Let $w' \in V'$ and consider the intersection $U_{w}' \cap U_{w'}'$. There we have $l'(w')= a l'(w)$ for some $a \in \CO_T\res{U_{w}'\cap U_{w'}'}$ and we must show that $l(w') = a l(w)$.
	It suffices to do this on the open subset $U_{w}' \cap U_{w'}' \cap U_{v}$, for arbitrary $v \in V$, as $U_{w}' \cap U_{w'}'$ is covered by these subsets. On $U_{w}' \cap U_{w'}' \cap U_{v}$ we have
	\begin{align*}
		l(w) = b l(v) ,& \qquad  l(w') = b' l(v),
	\end{align*}
	for some $b,b' \in \CO_T\res{U_{w}' \cap U_{w'}' \cap U_{v}}$.
	Then by the vanishing of $h^{-1} \CJ_{n-1} \cdot \CO_T$ it follows that
	\begin{align*}
		b' = \frac{l(w')}{l(v)} \frac{l'(w)}{l'(w)}= \frac{l(w)}{l(v)} \frac{l'(w')}{l'(w)} = b a ,
	\end{align*}
	and hence
	\[ l(w')= b' l(v)= b a l(v) = a l(w) . \]
\end{proof}

\begin{remarks}
	\begin{altitemize}
		\item[(i)]The morphism $\pi \colon \bB_V \lra \bP_V$ is given on $T$-valued points by
			\begin{align*}
				\bB_V (T) &\lra \bP_V (T) \\
				(\CL,l,\psi) &\longmapsto (\CL_V, l_V) 
			\end{align*}
			which follows from the compatibility of $\pi_i \colon \wt{\bP}_{i} \lra \wt{\bP}_{i-1}$ with the projection $\RX_i \lra \RX_{i-1}$ as seen in Proposition \ref{Prop - B Incidence variety X_i isomorphism}.
			Moreover, $\Omega_V \subset \bB_V$ represents the open subfunctor which associates to a $k$-scheme $T$ the set of all isomorphism classes $\left( \CL, l, \psi \right)$ such that $\psi_{W'}^{W}$ is an isomorphism of invertible sheaves, for all $\{0\} \neq W' \subset W \subset V$.
		\item[(ii)]In \cite[Thm.\ ~10.17]{PS14}, Pink and Schieder construct the following morphism $\bB_V \lra \bQ_V$ on the level of functors:
			For $\left( \CL, l, \psi \right) \in \bB_V (T)$, the homomorphisms $\psi_{W'}^{W} \colon \CL_{W'} \lra \CL_{W}$ give homomorphisms $\CL_{W}^{\vee} \lra \CL_{W'}^{\vee}$ of the duals.
			This makes $(\CL_{W}^{\vee})_{W\subset V}$ a direct system with direct limit 
			\[ \CF \defeq \varinjlim_{W \subset V} \CL_{W}^{\vee} \]
			which is an invertible sheaf.
			For $v \in V\setminus \{0\}$, consider the homomorphism
			\begin{alignat*}{2}
				\lambda_v \colon \CL_{k v} \overset{\substack{l_{kv} \\ \sim }}{\longleftarrow} k v &\otimes_k \CO_T \, &&\lra \CO_T			\\
																			v &\otimes a		&&\longmapsto a	.	 
			\end{alignat*}
			Then $\lambda_v \in \CL_{kv}^{\vee}$, and we define $r(v)$ as the image of $\lambda_v$ under the map $\Gamma(T, \CL_{kv}^{\vee}) \lra \Gamma (T, \CF)$. This yields a reciprocal map
			\[ r \colon V \setminus \{0\} \lra \Gamma (T, \CF) \]
			whose image are generating sections of $\CF$.
			On $\Omega_V$, this morphism agrees with $\rho \colon \bB_V \lra \bQ_V$, hence we find that the above is the description of $\rho $ on $T$-valued points.
	\end{altitemize}
\end{remarks}

\subsection{$\bar{k}$-valued points and stratification}

We proceed by relating the $\bar{k}$-valued points of $\bP_V$, $\bQ_V$, and $\bB_V$ with the stratifications of those varieties.

\begin{lemma}\label{Lemma - Points of P and stratification}
	Let $f \in \mathrm{\mathbf{P}}_V(\bar{k})$ correspond to $l \colon V_{\bar{k}}   \myrightarrowdbl{} \bar{k}$, and let $V' \subsetneq V$ be a subspace.
	Then $f \in \CP_{V/V'}(\bar{k})$ if and only if $\Ker( l ) \cap V = V'$.
\end{lemma}
\begin{proof}
	For any $W \subsetneq V$, we have $f \in \bP_{V/W}( \bar{k}) $ if and only if the inverse image ideal sheaf $f^{-1} \left( v \middle| v \in W \right)^{\sim} \cdot \bar{k}$ vanishes.
	Now
	\[f^{-1} \left( v \middle| v \in W \right)^{\sim} \cdot \bar{k} = \left( l(v ) \middle| v\in W \right), \]
	hence this is equivalent to $W \subset \Ker (l) \cap V$.
	
	Therefore $f \in \CP_{V/V'}(\bar{k})$ if and only if $V' \subset \Ker (l) \cap V$ and $W \nsubset \Ker (l) \cap V$, for all $W \supsetneq V'$.
	But this is equivalent to $V' = \Ker (l) \cap V$.
\end{proof}

\begin{lemma}\label{Lemma - Points of Q and stratification}
	Let $f \in \mathrm{\mathbf{Q}}_V (\bar{k})$ correspond to $r \colon V\setminus \{0\} \lra \bar{k}$, and let $\{0\}\neq V' \subset V$ be a subspace.
	Then $f \in \CQ_{V'}(\bar{k})$ if and only if 
	\[ r(v) = 0 \Leftrightarrow v \in V \setminus V' \, \text{, for all $v\in V\setminus \{0\}$.} \]
\end{lemma}
\begin{proof}
	For $\{0\}\neq W \subset V$, we have $f \in \bQ_{W}(\bar{k})$ if and only if the inverse image ideal sheaf $f^{-1} \left( \frac{1}{v} \middle| v \in V\setminus W \right)^{\sim} \cdot \bar{k}$ vanishes.
	Then
	\[ f^{-1} \left( \tfrac{1}{v} \middle| v \in V\setminus W \right)^{\sim} \cdot \bar{k} = \left( r(v) \middle| v \in V\setminus W \right) , \]
	therefore this is equivalent to $r(v)=0$ if $v \in V\setminus W$.
	
	Then $f \in \CQ_{V'}(\bar{k})$ if and only if $r(v) =0 $ for $v \in V\setminus V'$ and, for every $\{0\}\neq W \subsetneq V'$, there is some $v \in V \setminus W$ such that $r(v) \neq 0$.
	This in turn is equivalent to 
	\[ r(v) = 0 \Leftrightarrow v \in V \setminus V' \, \text{, for all $v \in V\setminus \{0\}$.} \]
\end{proof}

\begin{remark}\label{Rem - About Stratification of B}
	Let $\left( l_{W} \colon W_{\bar{k}} \myrightarrowdbl{} \bar{k}\right)_{W\subset V} \in \bB_{V}(\bar{k})$, and let $ \{0\} \neq W' \subset W \subset V$ be subspaces. Consider the commutative diagram:
	\[\begin{tikzcd}
		W_{\bar{k}} \arrow[r, two heads, "l_{W}"] 		& \bar{k} \arrow[d,leftarrow, "\psi_{W'}^{W}"] \\
		W'_{\bar{k}} \arrow[r, two heads, "l_{W'}"] \arrow[u, hook]	& \bar{k}\mathclap{\,\, .}
	\end{tikzcd}\]
	\begin{altenumerate}
		\item[(i) ] Then $\Ker (l_{W'}) \subset \Ker(l_W)$, and hence
		\[\Ker (l_{W'} ) \cap W' \subset \Ker ( l_{W}) \cap W. \]
		\item[(ii) ] As $\psi_{W'}^{W}$ is the multiplication with an element of $\bar{k}$ we have either $\psi_{W'}^{W} = 0$ or $\psi_{W'}^{W}$ is an isomorphism.
		Therefore, we have the two cases: either
		\[\Ker( l_{W} ) \cap W = W' , \]
		or there exists a representative of $\left( l_{W} \right)_{W\subset V}$ such that
		\[l_{W'} = l_{W}\res{W'_{\bar{k}}} .\]
	\end{altenumerate}
\end{remark}

\begin{lemma}\label{Lemma - Points of B and stratification}
	Let $ h \in \mathrm{\mathbf{B}}_V(\bar{k})$ be a $\bar{k}$-valued point corresponding to $( l_{W} \colon W_{\bar{k}} \myrightarrowdbl{} \bar{k} )_{W \subset V}$, and let $\CF= \left(\{0\} \subsetneq  V_{r} \subset\ldots\subset V_{0} \subsetneq V \right)$ be a flag.
	Then $h \in \CB_{\CF}(\bar{k})$ if and only if
	\[ \Ker ( l_{V_{i}} ) \cap V_{i} = V_{i+1} \, \text{, for all $i= -1,\ldots, r$,} \]
	where we write $V_{-1} = V$ and $V_{r+1} = \{0\}$.
\end{lemma}
\begin{proof}
	First, we show that $h$ is a $\bar{k}$-valued point of $D_{V'} \subset \bB_V$ if and only if for all subspaces $V' \subsetneq W \subset V$ we have $V' \subset \Ker( l_{W} ) \cap W$. Then we proceed similarly as in the proofs for $\bP_V$ and $\bQ_V$.

	By Proposition \ref{Prop - B Incidence variety X_i isomorphism},
	\[ b_{n-1}\left( D_{V'} \right) = \RX_{n-1} \cap  \Bigg(  \;\;\;\;  \prod\limits_{\mathclap{\substack{V' \subsetneq W \subset V}}} \; \bP_{W/V'} \times \bP_{V'} \times \;\; \prod\limits_{\mathclap{\substack{ \{0\} \neq W \subset V, \\ V' \nsubset W }}} \; \bP_{W} \Bigg) .\]
	Hence $h \in D_{V'}(\bar{k} )$ if and only if the inverse image ideal sheaf
	\[h^{-1} \Bigg( \;\;\;\; \sum\limits_{\mathclap{\substack{V'\subsetneq W \subset V}}} \;  \left( v \middle| v\in V' \right)^{\sim} \Bigg) \cdot \bar{k} \]
	vanishes (where for $V'\subsetneq W \subset V$, we extend $\left( v \middle| v\in V' \right)^{\sim} \subset \CO_{\bP_{W}}$ to $\CO_{\bP_{\ul{V},n-1}}$).
	But for each $V'\subsetneq W \subset V$ we find that
	\begin{align*}
		h^{-1}\left( \left( v \middle| v\in V'  \right)^{\sim} \right)\cdot \bar{k}
			&= f_{W}^{-1} \left( \left( v \middle| v\in V'  \right)^{\sim} \right)\cdot \bar{k}		\\
			&= \left( l_{W} ( v ) \middle| v \in V' \right) .
	\end{align*}
	This shows the first part.

	By the definition of $\CB_{\CF}$, $h\in \CB_{\CF}( \bar{k})$ if and only if the only subspaces $\{0\} \neq V' \subsetneq V$ with $h \in D_{V'}(\bar{k})$ are $V' =  V_{0},\ldots,V_{r}$.
	We show that this is equivalent to $\Ker( l_{V_{i}} ) \cap V_{i} = V_{i+1}$, for $i= -1,\ldots,r$:
	
	First we assume the latter.
	Let $i \in \{0,\ldots,r\}$ and consider for an arbitrary $V_{i}\subsetneq W \subset V$ the two cases of Remark \ref{Rem - About Stratification of B} (ii) applied to $W \subset V$.
	If $ l_W = l_V \res{W_{\bar{k}}}$, then
	\[V_{i}  \subset V_{0} \cap W  = \Ker( l_V ) \cap W = \Ker( l_{W}) \cap W .\]
	If $\Ker\left(l_V\right) \cap V = W$ it follows from $\Ker\left( l_V \right) \cap V = V_{0}$ that $W = V_{0} \supset V_{i}$.
	This cannot occur if $i= 0$.
	However, if $i > 0$, we deduce by Remark \ref{Rem - About Stratification of B} (i) that
	\[V_{i} = \Ker ( l_{V_{i-1}} ) \cap V_{i-1} \subset \Ker ( l_{V_{0}}) \cap V_{0} =  \Ker ( l_{W}) \cap W .\]
	In both cases, it follows by the first part that $h \in D_{V_{i}} ( \bar{k} )$.
	Furthermore, let $\{0\} \neq V' \subsetneq V$ such that $V' \notin \left\{V_{0},\ldots,V_{r}\right\}$.
	Let $i' \in \{-1,\ldots,r \}$ be maximal such that $V' \subset V_{i'}$.
	Note that this already implies $V' \subsetneq V_{i'}$.
	Then we have
	\[V' \nsubset V_{i'+1} = \Ker( l_{V_{i'}} ) \cap V_{i'} , \]
	hence $h\notin D_{V'} (\bar{k})$.
	
	Now assume that the only subspaces $\{0\} \neq V' \subsetneq V$ with $h \in D_{V'} (\bar{k})$ are $V' = V_{0},\ldots,V_{r}$.
	Let $i \in \{-1,\ldots,r\}$ and write $V_{i+1}' \defeq \Ker ( l_{V_{i}} ) \cap V_{i} \subsetneq V_{i}$.
	As $h \in D_{V_{i+1}} ( \bar{k})$, we surely have $V_{i+1} \subset V_{i+1}'$.
	Assume, by the way of contradiction, that $V_{i+1} \subsetneq V_{i+1}'$.
	Consider again, for arbitrary $V_{i+1}' \subsetneq W \subset V$, the two cases of Remark \ref{Rem - About Stratification of B} (ii) applied to $W \subset V$.
	If $ l_W = l_V \res{W_{\bar{k}}}$, then
	\[V_{i+1}' \subset \Ker ( l_{V_{i}}) \cap V_i \cap W \subset \Ker( l_{V}) \cap W = \Ker ( l_{W} ) \cap W .\]
	The case $\Ker (l_V ) \cap V = W$ cannot occur for $i=-1$.
	For $i > -1$, the claim for $i=-1$, i.e.\ $\Ker\left(l_V\right) \cap V = V_{0}$, implies that $W = V_0$. Hence, we conclude that
	\[V_{i+1}' = \Ker ( l_{V_{i}} ) \cap V_{i} \subset \Ker ( l_{V_{0}} ) \cap V_{0} = \Ker ( l_{W} ) \cap W . \]
	Like before, we have $ h \in D_{V_{i+1}'}( \bar{k} )$ in both cases.
	But this contradicts $V_{i+1}' \neq V_{0},\ldots,V_{r}$. Therefore, $\Ker ( l_{V_{i}} ) \cap V_{i} = V_{i+1}$.
\end{proof}

\section{$\PGL(V)$-action}\label{Sect - PGL(V)-action}

\subsection{Stabilizers and stratification}

We use the notation
\begin{align*}
	\GL(V)	&\defeq \mathbf{GL}(V) (k) ,	\\
	\PGL(V) &\defeq \mathbf{PGL}(V) (k) .
\end{align*}
Consider the canonical left $\GL(V)$-action on $k \! \left[ v \middle| v \in V \right] = \mathrm{Sym}\, V$:
\[ g\ldotp v = g(v) , \]
for all $g\in \GL(V), v\in V$. Each $g\in \GL(V)$ gives rise to an automorphism of $\bP_V$ by the rule
\[ \Fp\ldotp g  = g^{-1}(\Fp ) ,  \]
for all $g\in \GL(V)$, $ \Fp \in \bP_V$, so that we obtain a right $\PGL(V)$-action on $\bP_V$.
There is also the right $\GL(V)$-action on $ V_{\bar{k}}^{\ast}$ by the transpose
\[ g^{\ast}( l) ( v \otimes \lambda ) = (l \circ g )( v \otimes \lambda  ) = l ( g(v) \otimes \lambda )  , \]
for all $g\in \GL(V)$, $l \in  V_{\bar{k}}^{\ast}$ and $v \otimes \lambda \in V_{\bar{k}}$.
With the identification $\bP_V(\bar{k}) \cong \left( V_{\bar{k}}^{\ast} \setminus \{0\} \right) / \bar{k}\unts $ of Proposition \ref{Prop - Points of P} this gives the right $\PGL(V)$-action on $\bP_V(\bar{k})$ which corresponds to the right action on $\bP_V$ defined above.

For $\bQ_V$ consider the left $\GL(V)$-action on $k \! \left[ \frac{1}{v} \middle| v \in V\setminus \{0\} \right]$, given by
\[ g \ldotp \frac{1}{v} = \frac{1}{g(v)}, \]
for all $g\in \GL(V)$, $v \in V \setminus \{0\}$.
This induces a right $\PGL(V)$-action on $\bQ_V$ which on $\bar{k}$-valued points takes the form
\[ \left( r \ldotp g \right) ( v ) = (r \circ g)(v) = r( g(v) ) ,  \]
for all $g \in \PGL(V)$, $v\in V\setminus \{0\}$ and reciprocal maps $r \in \bQ_V (\bar{k})$.

Moreover, the induced right $\PGL(V)$-action on $\bB_V$, on $\bar{k}$-valued points, is given by
\[ \left( \left( l_{W} \colon W_{\bar{k}} \myrightarrowdbl{} \bar{k} \right)_{W \subset V} \right) \ldotp g = \left( l_{g(W)} \circ g\res{W_{\bar{k}}} \colon W_{\bar{k}} \myrightarrowdbl{} \bar{k} \right)_{W\subset V} , \]
for all $g \in \PGL(V)$ and $\left( l_{W} \right)_{W \subset V} \in \bB_V (\bar{k})$.
	
In subsection \ref{Subsect - Proof of Thm Stabilizers}, we will compute the stabilizers 
\[\mathrm{Stab}(x) = \left\{ g \in \PGL(V) \middle| x \ldotp g = x \right\}\subset \PGL(V) , \]
for $\bar{k}$-valued points $x$ of $\bP_V$, $\bQ_V$ and $\bB_V$ to arrive at the following theorem:

\begin{theorem}\label{Thm - Stabilizers}
	Let $g \in \PGL(V)$.
	\begin{altenumerate}
		\item[(i) ] Let $x \in \CP_{V/V'} (\bar{k}) \subset \mathrm{\mathbf{P}}_V (\bar{k})$.
			Then $g$ fixes $x$ if and only if $g$ is of the form
			\begin{align*}
				g = \left( \begin{array}{c|c}
						\ast & \ast \\
						\hline
						0	 & S
					\end{array} \right)
			\end{align*}
			with regard to $V = V' \oplus V''$ for some complementary space $V''$.
			Here $S \in \PGL(V'')$ is semi-simple, and specified by Proposition \ref{Prop - Fixpoints of Omega}.
			
		\item[(ii) ] Let $x \in \CQ_{V'} (\bar{k}) \subset \mathrm{\mathbf{Q}}_V (\bar{k})$.
			Then $g$ fixes $x$ if and only if $g$ is of the form
			\begin{align*}
				g = \left( \begin{array}{c|c}
						S & \ast \\
						\hline
						0	 & \ast
					\end{array} \right)
			\end{align*}
			with regard to $V = V' \oplus V''$ for some complementary space $V''$.
			Here $S \in \PGL(V')$ is semi-simple, and specified by Proposition \ref{Prop - Fixpoints of Omega}.
			
		\item[(iii) ] Let $x \in \CB_{\CF} (\bar{k}) \subset \mathrm{\mathbf{B}}_{V} (\bar{k})$, for a flag $\CF = \left( \{0\} \subsetneq V_r \subset\ldots\subset V_0 \subsetneq V \right)$.
		Then $g$ fixes $x$ if and only if $g$ is of the form
			\begin{align*}
				g = \left( \begin{array}{cccc}
						S_{r+1} &\multicolumn{1}{|c}{ } 	&						&  \ast					\\ \cline{1-2}	
							 	&\multicolumn{1}{|c}{S_{r}} &\multicolumn{1}{|c}{ } &  						\\ \cline{2-2}
								&							&\ddots					&  						\\ \cline{4-4}
						0		&							&						&\multicolumn{1}{|c}{S_0} 
					\end{array} \right)
			\end{align*}
			with regard to $V = V_{r} \oplus V_{r}' \oplus \ldots\oplus V_{0}'$ for some complementary spaces $V_{i-1} = V_{i} \oplus V_{i}'$, $i=0,\ldots,r$.
			The $S_{i}$ are again semi-simple, specified by Proposition \ref{Prop - Fixpoints of Omega}.
	\end{altenumerate}	
\end{theorem}

Recall that the parabolic subgroups of $\mathbf{PGL}(V)$ are the stabilizers of flags of $V$, and denote by $P_{\CF}$ the parabolic subgroup such that $P_{\CF}(k)$ stabilizes the flag $\CF$.
We say that a parabolic subgroup $P_{\CF}$ is \textit{maximal} if its flag consists of at most a single subspace, i.e.\ $\CF= ( V')$ or $P_{\CF} = \mathbf{PGL}(V)$.
Let $\FP$ denote the set of all parabolic subgroups of $\mathbf{PGL}(V)$, and let $\FP_{\max}$ be the set of all maximal parabolic subgroups.

Furthermore, let $R_\Ru G$ denote the unipotent radical of an algebraic group $G$.
If $G$ is a connected and reductive algebraic group over $k$, like in the case of $\mathbf{PGL}(V)$, then a parabolic subgroup $P \subset G$ is already determined by the $k$-rational points of its unipotent radical:
Indeed, let $P$ and $P'$ be two parabolic subgroups of $G$.
Then we have $P=P'$ if and only if $P(k)=P'(k)$ \cite[Cor. 4.20]{BT65}.
Moreover, $P(k)$ is uniquely characterized as the normalizer of $R_\Ru P (k)$ in $G(k)$ \cite[Cor. 5.19]{BT65}.
We write
\[R_\Ru \mathrm{Stab}(x) \defeq R_\Ru \mathrm{Stab}_{alg}(x)(k) . \]

\begin{corollary}\label{Cor - Stabilizers and Stratification}
	\begin{altitemize}
		\item[(i)]
			There is a bijection
			\[\FP_{\max} \overset{1:1}{\longleftrightarrow} \left\{ \text{strata of $\mathrm{\mathbf{P}}_V$} \right\} \]
			which is given by 
			\[\CP_{V/V'}(\bar{k}) = \left\{ x \in \mathrm{\mathbf{P}}_V (\bar{k})\middle|R_\Ru\mathrm{Stab}(x) = R_\Ru P_{( V' )} (k) \right\} ,\]
			for each maximal parabolic subgroup $P_{( V' )}$ of $\mathbf{PGL}(V)$.
		\item[(ii)]
		There is a bijection
			\[\FP_{\max} \overset{1:1}{\longleftrightarrow} \left\{ \text{strata of $\mathrm{\mathbf{Q}}_V$} \right\} \]
			which is given by 
			\[\CQ_{V'}(\bar{k}) = \left\{ x \in \mathrm{\mathbf{Q}}_V (\bar{k})\middle|R_\Ru\mathrm{Stab}(x)=R_\Ru P_{( V' )} (k) \right\} ,\]
			for each maximal parabolic subgroup $P_{( V' )}$ of $\mathbf{PGL}(V)$.
		\item[(iii)]
			There is a bijection
			\[\FP \overset{1:1}{\longleftrightarrow} \left\{ \text{strata of $\mathrm{\mathbf{B}}_V$} \right\} \]
			which is given by 
			\[\CB_{\CF}(\bar{k}) = \left\{ x \in \mathrm{\mathbf{B}}_V (\bar{k})\middle|R_\Ru\mathrm{Stab}(x)=R_\Ru P_{\CF} (k) \right\} ,\]
			for each parabolic subgroup $P_{\CF}$ of $\mathbf{PGL}(V)$.
	\end{altitemize}
\end{corollary}

\subsection{Proof of Theorem \ref{Thm - Stabilizers}}\label{Subsect - Proof of Thm Stabilizers}

First, note that $g\in \PGL(V)$ fixes $l \in \bP_V(\bar{k})$ if and only if for some representatives $g' \in \GL(V)$ and $l' \colon V_{\bar{k}} \myrightarrowdbl{} \bar{k}$ of $g$, respectively $l$, there exists $\lambda \in \bar{k}\unts$ such that $l'$ is an eigenvector of $g^{\prime \ast}$ in $V_{\bar{k}}^{\ast}$ for the eigenvalue $\lambda$, i.e.\
\[ g^{\prime \ast}(l') = \lambda l' .\]
This is equivalent to
\[ g'(v) - \lambda v \in \Ker(l') \, \text{, for all $v \in V_{\bar{k}}$.} \]
If $r \in \bQ_V (\bar{k})$, the condition for $g$ to fix $r$ in terms of a representative $r' \colon V \setminus \{0\} \lra \bar{k}$ is that there exists $\lambda \in \bar{k}\unts$ such that
\[ r'(g'(v)) = \lambda r'(v) \, \text{, for all $v \in V\setminus \{0\}$.} \]
For $(l_{W})_{W \subset V} \in \bB_V (\bar{k})$, one can express that $g$ fixes $(l_{W})_{W \subset V}$ in the analogous way.
In the following, we will implicitly choose representatives if necessary for concrete computations.

\begin{lemma}\label{Lemma - Fixpoint reduction P}
	Let $l \in \CP_{V/V'}(\bar{k})$, for a subspace $V'\subsetneq V$, and $g\in \PGL(V)$.
	Then $g$ fixes $l$ if and only if $V'$ is $g$-invariant, i.e.\ $\mathrm{\mathbf{P}}_{V/V'}$ is $g$-invariant, and the induced $\bar{g}\in \PGL(V/V')$ fixes $\bar{l} \in \mathrm{\mathbf{P}}_{V/V'}(\bar{k})$.
\end{lemma}
\begin{proof}
	For $v\in V_{\bar{k}}$ we denote the residue in $\left( V/V'\right)_{\bar{k}}$ by $\ov{v}$.
	First let $g$ fix $l$, i.e.\ let there be $\lambda \in \bar{k}\unts$ such that $g(v) - \lambda v \in \Ker(l)$ for all $v\in V_{\bar{k}}$.
	As $l \in \CP_{V/V'}(\bar{k})$, it follows from Lemma \ref{Lemma - Points of P and stratification} that $V' = \Ker(l) \cap V \subset \Ker(l)$.
	Hence $g(v) \in \Ker(l) \cap V = V' $, for all $v\in V'$.
	Therefore
	\[ \bar{g} \colon V/V' \lra V/V', \; \ov{v} \longmapsto \ov{g(v)} \]
	is a well defined automorphism of $V/V'$, and for $\bar{l} \colon \left( V/V'\right)_{\bar{k}} \myrightarrowdbl{} \bar{k}$ we compute that
	\[ \bar{l}\left( \bar{g}(\ov{v}) - \lambda \ov{v}  \right) = \bar{l}\left( \ov{g(v) - \lambda v} \right) = l\left( g(v)-\lambda v\right) = 0 , \]
	for all $\ov{v} \in \left(V/V'\right)_{\bar{k}}$.
	Hence $\bar{g}$ fixes $\bar{l}\in \bP_{V/V'}(\bar{k})$.
	
	Conversely, let $V'$ be $g$ invariant and let $\bar{g}$ fix $\bar{l}\in \bP_{V/V'}(\bar{k})$, i.e.\ let $\bar{g}(\ov{v}) - \lambda\ov{v} \in \Ker(\bar{l})$, for all $ \ov{v} \in \left(V/V'\right)_{\bar{k}}$.
	As 
	\[\bar{l}\left( \ov{g(v) - \lambda v }\right) = \bar{l}\left( \bar{g}(\ov{v}) - \lambda \ov{v}  \right) = 0 ,  \]
	for all $v \in V_{\bar{k}}$, we have $g(v)-\lambda v \in \Ker(l) + V_{\bar{k}}' = \Ker(l)$.
	Therefore $g$ fixes $l$.	
\end{proof}

\begin{lemma}\label{Lemma - Fixpoint reduction Q}
	Let $r \in \CQ_{V'}(\bar{k})$, for a subspace $\{0\} \neq V' \subset V$, and $g\in \PGL(V)$.
	Then $g$ fixes $r$ if and only if $V\setminus V'$ is $g$-invariant, i.e.\ $\mathrm{\mathbf{Q}}_{V'}$ is $g$-invariant, and $g\res{V'}$ fixes $r\res{V'\setminus \{0\} } \in \mathrm{\mathbf{Q}}_{V'}(\bar{k})$.
\end{lemma}
\begin{proof}
	Let $g$ fix $r$, i.e.\ let $\lambda \in \bar{k}\unts$ such that
	\[ r(g(v))  =  \lambda r(v) \, \text{, for all $v \in V\setminus \{0\}$.} \]
	Hence, by Lemma \ref{Lemma - Points of Q and stratification}, we deduce that both $V\setminus V'$ and $V'$ are $g$-invariant.
	Furthermore, it follows directly that $g\res{V'}$ fixes $r\res{V' \setminus \{0\}}$.
	
	Conversely, suppose that $V\setminus V'$ is $g$-invariant and $g\res{V'}$ fixes $r\res{V' \setminus \{0\}}$.
	Let $\lambda \in \bar{k}\unts $ such that
	\[ r\res{V' \setminus \{0\}} (g\res{V'} (v)) = \lambda r\res{V' \setminus \{0\}} (v) \, \text{, for all $v\in V'\setminus \{0\}$. } \]
	But for $v\in V\setminus V'$, we have $g(v) \in V\setminus V'$, hence by Lemma \ref{Lemma - Points of Q and stratification}
	\[ r(g(v)) = 0 = \lambda r(v) \, \text{, for all $v\in V \setminus V'$.} \]
	Therefore $g$ fixes $r$.
\end{proof}

The following lemma allows us to determine the stabilizer of $\bar{k}$-valued points of $\CQ_V$ by computing the stabilizer of $\bar{k}$-valued points of $\CP_V$:

\begin{lemma}\label{Lemma - Equivariant map k points Omega}
	The isomorphism of $\bar{k}$-valued points
	\begin{alignat*}{2}
		\mathrm{\mathbf{P}}_V (\bar{k}) \supset \Omega_V (\bar{k})	&	\,\, \;\;\!\!\! =	&& \;	\Omega_V (\bar{k})	\subset \mathrm{\mathbf{Q}}_V (\bar{k})	\\
			\left( l \colon V_{\bar{k}} \twoheadrightarrow \bar{k} \right)		&\longmapsto	&& \, \left( r \colon V\setminus \{0\} \rightarrow \bar{k}, v \mapsto \tfrac{1}{l(v)} \right)
	\end{alignat*}
	induced from the birational equivalence of $\mathrm{\mathbf{P}}_V $ and $ \mathrm{\mathbf{Q}}_V$ is equivariant with respect to the action of $\PGL(V)$.
\end{lemma}

The proof of this lemma is a straightforward calculation, and we omit it.

\begin{lemma}\label{Lemma - Fixpoint reduction B}
	Let $\left( l_{W} \right)_{W\subset V} \in \CB_{\CF}( \bar{k} )$, for some flag $\CF=\left( \{0\} \subsetneq V_{r} \subset \ldots \subset V_{0} \subsetneq V \right)$, and $g\in \PGL(V)$.
	Then $g$ fixes $\left( l_{W} \right)_{W \subset V}$ if and only if $g$ fixes the flag $\CF$, i.e.\ $\bigcap_{V' \in \CF} D_{V'}$ is $g$-invariant, and the induced $\bar{g}_i \in \PGL\left(V_{i-1}/V_{i}\right)$ fixes the induced $\bar{l}_{V_{i-1}} \in \mathrm{\mathbf{P}}_{V_{i-1}/V_{i}} ( \bar{k} ) $, for all $i=0,\ldots,r+1$, where $V_{-1} = V$ and $V_{r+1} = \{0\}$.
\end{lemma}
\begin{proof}
	Recall that for $\left( l_{W} \right)_{W\subset V} \in \CB_{\CF}( \bar{k} )$ by Lemma \ref{Lemma - Points of B and stratification}
	\[ \Ker ( l_{V_{i}} ) \cap V_{i} = V_{i+1} , \]
	for all $i= -1,\ldots, r$.
	The element $g \in \PGL(V)$ fixes $\left( l_{W} \right)_{W\subset V} \in \bB_V ( \bar{k} )$ if and only if for all $\{0\} \neq W \subset V$, there exist $\lambda_{W} \in \bar{k}\unts$ such that
	\begin{align}\label{Eq - Fixing for subspaces B}
		l_{g(W)} \circ g\res{W_{\bar{k}}}  = \lambda_{W} l_{W}.
	\end{align}
	
	First, assume that this is the case.
	By applying Lemma \ref{Lemma - Fixpoint reduction P} to $l_V$, we find that $g$ leaves $V_{0} = \Ker(l_V) \cap V$ invariant, and $\bar{g}_{0} \in \PGL (V/V_{0} )$ fixes $\bar{l}_{V} \in \bP_{V/V_{0}} ( \bar{k} )$.
	Then \eqref{Eq - Fixing for subspaces B} reads for $V_{0}$:
	\[l_{V_{0}} \circ g\res{V_{0 \bar{k}}} = \lambda_{V_{0}} l_{V_{0}} .\]
	Hence we can proceed inductively to arrive with the claim.
	
	Now suppose that $g$ fixes $\CF$, and $\bar{g}_i \in \PGL\left(V_{i-1}/V_{i}\right)$ fixes $\bar{l}_{V_{i-1}} \in \bP_{V_{i-1}/V_{i}} ( \bar{k} ) $, for all $i=0,\ldots,r+1$.
	By Lemma \ref{Lemma - Fixpoint reduction P} this implies that $g$ fixes $l_{V_{i}}$ with some $\lambda_{V_{i}} \in \bar{k}\unts$, for $i=-1,\ldots,r$.
	For an arbitrary subspace $\{0\} \neq W \subset V$, let $i' \in \{0,\ldots,r+1 \}$ such that $W \subset V_{i'-1}$ and $W \nsubset V_{i'}$.
	We apply Remark \ref{Rem - About Stratification of B} (ii) to $W \subset V_{i'-1} $:
	The first case cannot occur, as 
	\[ V_{i'} = \Ker ( l_{V_{i'-1}}  ) \cap V_{i'-1} = W \]
	is a contradiction to $W \nsubset V_{i'}$. 
	Hence it must be $l_{W} = l_{V_{i-1}}\res{W_{\bar{k}}}$. 
	Note that also $g(W) \subset V_{i'-1}$ and $g(W) \nsubset V_{i'}$ as $V_{i'-1}$ and $V_{i'}$ are $g$-invariant.
	Therefore, the same reasoning applies, and we have $l_{g(W)} = l_{V_{i'-1}}\res{g(W)_{\bar{k}}}$.
	We compute that
	\begin{align*}
		l_{g(W)} \circ g\res{W_{\bar{k}}} &= l_{V_{i'-1}}\res{g(W)_{\bar{k}}} \circ g\res{W_{\bar{k}}}	\\
			&= (l_{V_{i'-1}} \circ g ) \res{W_{\bar{k}}}				\\
			&= \lambda_{V_{i'-1}} l_{V_{i'-1}}\res{W_{\bar{k}}}			\\
			&= \lambda_{V_{i'-1}} l_{W} .
	\end{align*}
	Hence $g$ fixes $\left( l_{W} \right)_{W\subset V}$.
\end{proof}

Regarding the proof of Theorem \ref{Thm - Stabilizers}, it remains to consider $\mathrm{Stab}(l)$ for $l \in \Omega (\bar{k}) \subset \bP_V (\bar{k})$.
We have the canonical left $\Gal(\bar{k}/k)$-action on $V_{\bar{k}}$ by
\[ \sigma  \left( v \otimes \lambda \right) = v \otimes \sigma(\lambda) , \]
for all $\sigma \in \Gal(\bar{k}/k)$, $v \otimes \lambda \in V_{\bar{k}}$.
This induces also a right $\Gal(\bar{k}/k)$-action on $V_{\bar{k}}^{\ast}$ with
\[ l^{\sigma} (v)  =  \sigma^{-1} \circ l \circ \sigma (v) , \]
for all $\sigma \in \Gal(\bar{k}/k)$, $l \in V_{\bar{k}}^{\ast}$, $v \in V_{\bar{k}}$, and a right $\Gal (\bar{k}/k)$-action on $\GL(V_{\bar{k}})$ by
\[ g^{\sigma} (v) = \sigma^{-1} \circ g \circ \sigma (v) , \]
for all $\sigma \in \Gal(\bar{k}/k)$, $g \in \GL(V_{\bar{k}})$, $v \in V_{\bar{k}}$.
The elements of $\GL(V)$ are exactly the elements of $\GL(V_{\bar{k}})$ which are fixed under this $\Gal(\bar{k}/k)$-action.	
Furthermore, let $F \in \Gal(\bar{k}/k)$ denote the Frobenius automorphism relative to $k = \BF_q$, with $F(\lambda) = \lambda^q$, for $\lambda \in \bar{k}$.

\begin{lemma}\label{Lemma - Basis for points of Omega}
	Let $ l \in \CP_{V/V'}( \bar{k} )$. Then $dim_{\bar{k}} \, \Span\left( l,l^{F},\ldots \right) = \dim_k \, V - \dim_k \, V'$.
	Especially if $ l \in \CP_{V}( \bar{k} ) = \Omega_V (\bar{k})$, then $l,l^{F},\ldots,l^{F^{n}}$ is a basis of $ V_{\bar{k}}^{\ast}$.
\end{lemma}
\begin{proof}
	Let $U \defeq \Span\left( l,l^{F},\ldots \right) \subset  V_{\bar{k}}^{\ast}$. Then we have 
	\[ U \cong \bigg( V_{\bar{k}} \bigg/  \bigcap_{i\geq 0} \Ker \big( l^{F^{i}} \big)  \bigg)^{\ast} \]
	and 
	\[\bigcap_{i\geq 0} \Ker\big( l^{F^{i}} \big) = \bigg( \bigcap_{i\geq 0} \left( \Ker \big( l^{F^{i}} \big) \cap V \right) \bigg)_{\bar{k}} \]
	as $\bigcap_{i\geq 0} \Ker \big(l^{F^{i}}\big) $ is $\Gal(\bar{k}/k)$-invariant. Furthermore
	\[\bigcap_{i\geq 0} \left(\Ker \big( l^{F^{i}} \big) \cap V \right) = \Ker(l) \cap V  \]
	because for $v\in \Ker(l) \cap V$ we have
	\[l^{F^i} (v) = F^{-i}\circ l \circ F^i (v)= F^{-i}\circ l(v) = 0 , \]
	for all $i\geq 0$.
	Since $ l \in \CP_{V/V'} (\bar{k})$, we have $\Ker(l) \cap V = V'$, and the first claim follows.
	
	If $V'= \{0\}$, this implies that $U =  V_{\bar{k}}^{\ast}$.
	Moreover, $l,l^{F},\ldots,l^{F^{n}}$ are linearly independent, as otherwise one finds that $\dim_{\bar{k}} \, U < \dim_k \, V =  n+1$.
\end{proof}

Let $k_d$ denote the unique intermediate field $k \subset k_d \subset \bar{k}$ of degree $d$ over $k$, i.e.\ $k_d$ consists of the elements in $\bar{k}$ that are fixed by $F^d \in \Gal(\bar{k}/k)$.

\begin{proposition}\label{Prop - Fixpoints of Omega}
	Let $l \in \Omega_V (\bar{k}) \subset \mathrm{\mathbf{P}}_{V} (\bar{k}) $ and $g \in \PGL(V)$. Then $g$ fixes $l$ if and only if the following conditions hold:
	\begin{altitemize}
		\item[(i) ] There exists $d\in \BN$ with $d \mid \dim_k \, V = n+1 $ such that
			\[ \Span \left( l,l^{F^d},\ldots \right) = \Span \left( l, l^{F^d},\ldots, l^{F^{n+1-d}} \right) \eqdef U . \]
		\item[(ii )] There exists $\lambda \in k_d\unts$ such that the transpose $g^{\ast} \in \PGL(V^{\ast}_{\bar{k}})$ is of the form
			\[ g^{\ast} \res{F^i (U)} = F^{-i} (\lambda) \, \mathrm{id}_{F^i (U)} \, \text{, for $i= 0,\ldots,d-1$.} \]
	\end{altitemize}
\end{proposition}

\begin{remarks}\label{Rem - For fixpoints of Omega}
	\begin{altitemize}
		\item[(i)] Especially, for $l \in \Omega_V (\bar{k})$, all elements in $\mathrm{Stab}(l) \subset \PGL(V)$ are semi-simple.
		\item[(ii)] Any point $ l \in \bP_V(\bar{k})$ can be viewed as a $\bar{k}$-valued point of $\bP_{V_{k_d}}$,
			and lies in a stratum $\CP_{V_{k_d}/V'}(\bar{k}) \subset \bP_{V_{k_d}}(\bar{k})$ of dimension less or equal to the dimension of its stratum in $\bP_V$.
			For $l \in \Omega_V (\bar{k}) \subset \bP_V (\bar{k})$, the dimension of this $\CP_{V_{k_d}/V'} \subset \bP_{V_{k_d}}$ is always
			\[ \dim \, \CP_{V_{k_d}/V'} \geq \frac{n+1}{d} -1 . \]	
			Then the equivalent geometric interpretation of condition (i) on $l$ in the Proposition \ref{Prop - Fixpoints of Omega} is that
			\[ \dim \, \CP_{V_{k_d}/V'} = \frac{n+1}{d} -1, \]
			i.e.\ the dimension of the stratum of $l \in \bP_{V_{k_d}}(\bar{k})$ is minimal.
	\end{altitemize}
\end{remarks}

\begin{proof}[Proof of Remark \ref{Rem - For fixpoints of Omega} (ii)]
	Assume $l \in \Omega_V (\bar{k}) \subset \bP_V (\bar{k})$, and let $l \in \CP_{V_{k_d}/V'}(\bar{k}) \subset \bP_{V_{k_d}}(\bar{k})$ for some subspace $V' \subset V_{k_d}$.
	By Lemma \ref{Lemma - Points of P and stratification} applied to $\bP_{V_{k_d}}$, we have $V' = \Ker(l) \cap V_{k_d}$.
	By using Lemma \ref{Lemma - Basis for points of Omega} for $l \in \CP_{V_{k_d}/V'}(\bar{k})$ and the Frobenius $F^d$ of $k_d$, we find that
	\[ \dim_{\bar{k}} \, \Span \left( l, l^{F^d},\ldots \right) = n+1 - \dim_{k_d} \, V' . \]
	But as $l \in \Omega_V (\bar{k})$, the elements $l,l^{F^d},\ldots, l^{F^{n+1-d}}$ are $\bar{k}$-linearly independent, hence
	\[ \dim_{\bar{k}} \, \Span \left( l, l^{F^d},\ldots \right) \geq \dim_{\bar{k}} \, \Span \left( l, l^{F^d},\ldots , l^{F^{n+1-d}}\right) = \frac{n+1}{d}. \]
	This yields:
	\begin{align}\label{Eq - Dimension stratum induced by Omega field extension}
		\dim \, \CP_{V_{k_d}/V'} = n - \dim_{k_d} \, V' \geq \frac{n+1}{d} - 1 .
	\end{align}
	Furthermore, we see that condition (i) in Proposition \ref{Prop - Fixpoints of Omega} is equivalent to equality in \eqref{Eq - Dimension stratum induced by Omega field extension}.	
\end{proof}

\begin{proof}[Proof of Proposition \ref{Prop - Fixpoints of Omega}]
	First, let $g$ fix $l$.
	Let $\lambda \in k_d \subset \bar{k}$, $d \in \BN$ minimal, such that $l$ is an eigenvector of $g^{\ast}$ for the eigenvalue $\lambda$, i.e.\
	\[ g^{\ast}(l)(v) = l(g(v)) = l(\lambda v)= \lambda l(v) , \]
	for all $v\in V_{\bar{k}}$.
	Moreover, as $g$ is $\Gal(\bar{k}/k)$-invariant, we find that $l^{F^i}$ is an eigenvector of $g^{\ast}$ for the eigenvalue $F^{-i} (\lambda)$, for $i\geq 0$:
	\begin{align*}
		g^{\ast}\big( l^{F^i} \big) (v) &= \big( F^{-i} \circ l \circ F^i \big) \big( g(v) \big)	\\
			&= \big( F^{-i} \circ l \big) \big(  g \big( F^i (v) \big) \big)						\\	
			&= \big( F^{-i} \circ l \big) \big( \lambda F^i (v) \big) 								\\
			&=  F^{-i}(\lambda) \, l^{F^i} (v).
	\end{align*}
	Hence $V^{\ast}_{\bar{k}}$ has a basis of eigenvectors $l,l^F,\ldots,l^{F^{n}}$, by Lemma \ref{Lemma - Basis for points of Omega}.
	Because $\lambda \in k_d$, with $d$ minimal, $F^{-d}(\lambda)= \lambda$, and $F^{-i} (\lambda) \neq \lambda $ for $ i = 1,\ldots,d-1$.
	But since all $\Gal(\bar{k}/k)$-conjugates of $\lambda$ are eigenvalues of $g^{\ast}$, it follows that $d\leq n+1$.
	Furthermore, by applying $F^{-i}$ to the eigenspace for the eigenvalue $F^{-i}(\lambda)$ we see that all eigenspaces have the same dimension.
	Hence $d \mid (n+1)$, and the eigenspace for $\lambda$ is
	\[U \defeq \Span\left(l,l^{F^d},\ldots,l^{F^{n+1-d}}\right) . \]
	Therefore, the eigenvalues of $g^{\ast}$ are in fact $\lambda,F^{-1}(\lambda),\ldots,F^{-d+1}(\lambda)$, with corresponding eigenspaces $U, F(U),\ldots,F^{d-1}(U)$.
	It follows that $g^{\ast}$ fulfils condition (ii) of the Proposition. 
	Furthermore, by applying $F^{di}$ to $l$, for any $i \geq 0$, we find that $l^{F^{di}}$ is an eigenvector for $\lambda$.
	Hence $l^{F^{di}} \in U$, and condition (i) follows.
	
	Conversely, suppose $l$ and $g$ satisfy the two conditions.
	By condition (ii), $l$ is an eigenvector of $g^{\ast} \in \PGL( V^{\ast}_{\bar{k}})$, i.e.\ $g$ fixes $l$.
	We are left to show that the definition of $g^{\ast}\in \PGL( V^{\ast}_{\bar{k}})$ by 
	\[g^{\ast} \res{F^i (U)} = F^{-i} (\lambda) \, \mathrm{id}_{F^i (U)} \, \text{, for $i= 0,\ldots,d-1$,}\]
	already implies $g\in \PGL(V)$.
	It suffices to verify that $\big(g^{F^{-1}}\big)^{\ast} = g^{\ast}$, and this may be done on the basis $l,l^F,\ldots, l^{F^n}$ of $V^{\ast}_{\bar{k}}$.
	For $i=0,\ldots,n-1$, we compute that
	\begin{align*}
		\big(g^{F^{-1}}\big)^{\ast} \big( l^{F^i} \big)	&= \big( F^{-i} \circ l \circ F^i \big)\circ\big(  F \circ  g \circ F^{-1} \big)	\\
														&= F \circ F^{-i-1} \circ l \circ F^{i+1} \circ g \circ F^{-1}						\\
														&= F \circ l^{F^{i+1}} \circ g \circ F^{-1}											\\
														&= F \circ \big( F^{-i-1}(\lambda) \, l^{F^{i+1}} \big) \circ F^{-1}					\\
														&= F^{-i} \,(\lambda) l^{F^i}
	\end{align*}
	which agrees with 
	\[ g^{\ast} \big( l^{F^i} \big) = F^{-i}\, (\lambda) l^{F^i} . \]
	Furthermore, we have $l^{F^{n+1}} \in U$ by condition (i), hence
	\[ g^{\ast} \big( l^{F^{n+1}} \big) = \lambda \, l^{F^{n+1}}. \]
	Therefore, the computation for $i=n$, analogous to the above, shows that
	\[ \big(g^{F^{-1}}\big)^{\ast} \big( l^{F^n} \big) = F^{-n} \, (\lambda) l^{F^n} = g^{\ast} \big( l^{F^n} \big) . \]
\end{proof}

\begin{proof}[Proof of Theorem \ref{Thm - Stabilizers}]
	Let $x \in \CP_{V/V'} (\bar{k})$ correspond to $l \colon V_{\bar{k}} \myrightarrowdbl{} \bar{k}$.
	Let $V''$ be a complementary space of $V'$, and let 
	\begin{align*}
		g = \left( \begin{array}{c|c}
						A & B \\
						\hline
						C & D
					\end{array} \right)
	\end{align*}
	with respect to the decomposition $V= V' \oplus V''$.
	Then by Lemma \ref{Lemma - Fixpoint reduction P}, $g$ fixes $x$ if and only if $C=0$ and the residue $\ov{D}$ of $D$ is an element of $\mathrm{Stab}(\bar{l}) \subset \PGL(V/V')$.
	But $\bar{l} \in \Omega_{V/V'} (\bar{k})$, hence Proposition \ref{Prop - Fixpoints of Omega} applies to $\ov{D}$ and $\bar{l}$.
	
	Now, let $x \in \CQ_{V'}(\bar{k})$ correspond to $r \colon V\setminus \{0\} \lra \bar{k}$.
	Let again $V''$ be a complementary space of $V'$, and consider $g$ in the above form with respect to $V= V' \oplus V''$.
	By Lemma \ref{Lemma - Fixpoint reduction Q}, $g$ fixes $x$ if and only if $V\setminus V'$ is $g$-invariant and $A$ lies in $\mathrm{Stab}(r\res{V'\setminus \{0\}}) \subset \PGL(V')$.
	Note that $V\setminus V'$ is $g$-invariant if and only if $V'$ is $g$-invariant which in turn is equivalent to $C=0$.
	Furthermore, by Lemma \ref{Lemma - Equivariant map k points Omega}, $A$ fixes $r\res{V'\setminus \{0\}}$ if and only if $A$ fixes $l' \in \CP_{V'} (\bar{k})$ where
	\[ l' \colon V^{\prime}_{\bar{k}} \myrightarrowdbl{} \bar{k} , \; v \longmapsto \frac{1}{r(v)} . \]
	Hence, we can apply Proposition \ref{Prop - Fixpoints of Omega} to $A$ and $l'$.
	
	Let $x \in \CB_{\CF} (\bar{k})$ correspond to $(l_W )_{W\subset V}$.
	Let $g$ be of the form
	\begin{align*}
		g = \left( \begin{array}{cccc}
						A_{r+1,r+1} &\multicolumn{1}{|c}{A_{r+1,r}} &\multicolumn{1}{|c}{\cdots}&\multicolumn{1}{|c}{A_{r+1,0}}\\ \cline{1-2} \cline{4-4}	
					 	A_{r,r+1}	&\multicolumn{1}{|c}{A_{r,r}}	&\multicolumn{1}{|c}{\cdots}&\multicolumn{1}{|c}{A_{r,0}}	\\ \cline{1-2} \cline{4-4}
						\vdots		&\vdots							&\ddots						&\vdots		\\ \cline{1-2} \cline{4-4}
						A_{0,r+1}	&\multicolumn{1}{|c}{A_{0,r}}		&\multicolumn{1}{|c}{\cdots}	&\multicolumn{1}{|c}{A_{0,0}} 
			\end{array} \right)
	\end{align*}
	with respect to $V = V_{r} \oplus V_{r}' \oplus \ldots\oplus V_{0}'$, for complementary spaces $V_{i-1} = V_{i} \oplus V_{i}'$, $i=0,\ldots,r$.
	By Lemma \ref{Lemma - Fixpoint reduction B}, $g$ fixes $x$ if and only if $g$ stabilizes the flag $\CF$, i.e.\ 
	\[ A_{i,j} = 0 \, \text{, for $j>i$,} \]
	and the residue $\ov{A_{i,i}}$ of $A_{i,i}$ is an element of $\mathrm{Stab}(\bar{l}_{V_{i-1}}) \subset \PGL(V_{i-1}/V_i)$.
	As $\bar{l}_{V_{i-1}} \in \Omega_{V_{i-1}/V_i}$, again Proposition \ref{Prop - Fixpoints of Omega} applies to $\ov{A_{i,i}}$ and $\bar{l}_{V_{i-1}}$, for $i=0,\ldots r+1$.
\end{proof}

\end{document}